\documentclass[a4paper,10pt,reqno]{amsart}

\usepackage{amsmath,amsfonts,amsthm,amssymb,dsfont,mathabx}
\usepackage[alphabetic]{amsrefs}
\usepackage[OT4]{fontenc}
\usepackage{enumerate}
\usepackage{mathrsfs,mathtools}
\usepackage{enumerate, xspace}
\usepackage[pdftex]{graphicx}
\usepackage{color}
\usepackage{float}
\usepackage[center, font=small,labelfont=sc,labelsep=none]{caption}

\long\def\symbolfootnote[#1]#2{\begingroup
\def\thefootnote{\fnsymbol{footnote}}\footnote[#1]{#2}\endgroup}

\newcommand{\ignore}[1]{}

\newtheorem{theorem}{Theorem}[section]
\newtheorem{lemma}[theorem]{Lemma}

\newtheorem{prop}[theorem]{Proposition}

\newtheorem{main}{Theorem}

\theoremstyle{definition}
\newtheorem{rem}[theorem]{Remark}

\newcommand{\executeiffilenewer}[3]{%
\ifnum\pdfstrcmp{\pdffilemoddate{#1}}%
{\pdffilemoddate{#2}}>0%
{\immediate\write18{#3}}\fi%
}

\begin{document}

\title[Iterated Medial Triangle Subdivision]{Iterated Medial Triangle Subdivision in Surfaces of Constant Curvature}

\author[F.~Brunck]{Florestan Brunck}
\address{McGill University\\
 Burnside Hall, 805 Sherbrooke Street West, H3A 0B9, Montreal, QC, Canada}
\email{florestan.brunck@mail.mcgill.ca}

\begin{abstract}
\noindent
Consider a geodesic triangle on a surface of constant curvature and subdivide it recursively into 4 triangles by joining the midpoints of its edges. We show the existence of a uniform $\delta>0$ such that, at any step of the subdivision, all the triangle angles lie in the interval $(\delta, \pi -\delta)$. Additionally, we exhibit stabilising behaviours for both angles and lengths as this subdivision progresses.
\end{abstract}

\maketitle

\section{Introduction}
\label{sec1}
In this section, we motivate the study of this subdivision in a non-Euclidean setting, introduce some notation, and state our main results. 

We first define what we call the \textit{iterated medial triangle subdivision} (see Figure \ref{fig:example}). In our setting, all geodesics will be taken to be minimal. A \textit{geodesic triangle} $T$ in the surface $M^2_\kappa$ of constant curvature $\kappa$ is defined as a triple of points of $M^2_\kappa$, together with a choice of three geodesic segments joining each pair of points. If $\kappa \leq 0$, $M^2_\kappa$ is uniquely geodesic and any triple of points defines a unique geodesic triangle. If $\kappa > 0$ however, there exists a unique geodesic between two points if and only if the distance between them is strictly less than $\frac{\pi}{\sqrt{\kappa}}$ (\cite{BH}). For our subdivision to be well-defined in the positive curvature case, we then require the three vertices of our triangle to lie in the same open hemisphere (the largest uniquely geodesic convex set in $M_\kappa^2$, see \cite{BH}). Equivalently, we could require the perimeter of our triangles to be strictly less than $2\pi$. In the positive curvature setting, we shall then understand the meaning of ``geodesic triangle'' to include these restrictions on the possible triples of points. For all $\kappa$, we define the iterated medial triangle subdivision of a geodesic triangle $T\subset M_\kappa^2$ inductively, as the following sequence $T_0, T_1, T_2, \ldots$ of refining triangulations:
\begin{itemize}
    \item $T_0=T$
    \item $T_{n+1}$ is obtained from $T_n$ by adding the midpoints of the edges of $T_n$ and, within each triangle of $T_n$, pairwise connecting its 3 midpoints by geodesic segments  (this creates $4$ new sub-triangles for each triangle of $T_n$). 
\end{itemize}

\begin{figure}[H]
	\begin{center}     
    	\includegraphics[width=\textwidth]{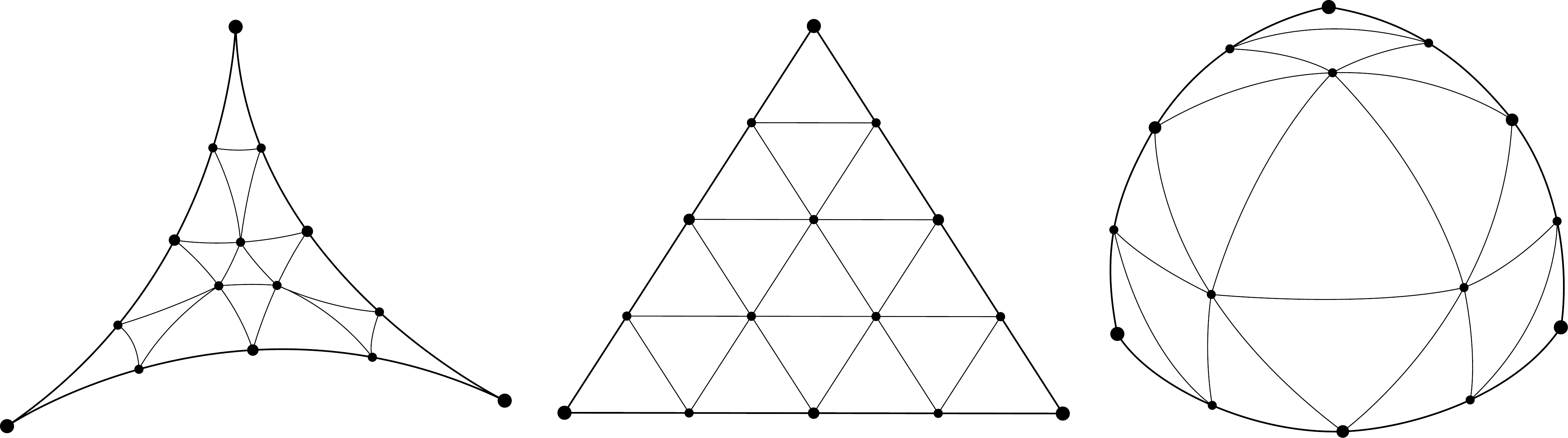}
        	\caption{: The first two medial triangle subdivisions of a triangle in $\mathbb H^2$, $\mathbb E^2$ and $\mathbb S^2$.}
        	\label{fig:example}
	\end{center}
\end{figure}

Our main work is to show that this elementary subdivision behaves ``nicely'' with respect to both lengths and angles, in a sense made precise by our three main theorems, A, B and C. Before stating our theorems, we point out that it is enough to prove each theorem in the specific case where $\kappa=\pm 1$. The general case can be reduced to the previous one by rescaling the spherical/hyperbolic metric by a constant $|\kappa|^{-\frac{1}{2}}$. Indeed, doing so the curvature becomes $\kappa$, but angles are not affected.

\begin{main}
\label{thm:main}
For any geodesic triangle T in $M_\kappa^2$, there exists $\delta>0$ such that, for all $n$, all the angles of $T_n$ lie in the interval $(\delta, \pi -\delta)$.
\end{main}

This theorem will be derived as an immediate consequence of our Theorem B. However, we point out to the reader that Theorem A does not require the full strength of Theorem B and can be obtained through a faster route (see Remark \ref{alternativeroute}). Before stating our remaining theorems, we first need to introduce some notation to make clear the statements of Theorems B and C.

To fix our notation, we will consider a sequence of \textit{nested} triangles $t_0, t_1, \ldots$ with $t_i\in T_i, i\in \mathbb N$, such that $t_{n+1}$ is one of the 4 triangles of the medial triangle subdivision of $t_n$. We name the sides of $t_0$ as $a_0$,  $b_0$, and $c_0$ and use the following notation scheme, by analogy with the Euclidean case (c.f. Figure \ref{fig:sequence}):

\begin{enumerate}[(a)]
    \item If $t_{n+1}$ is obtained as the innermost triangle in the medial decomposition of $t_n$ (as seen on the left diagram), then we name each of its edges according to the only edge of  $t_{n}$ it does not intersect, e.g. $a_{n+1}$ denotes the side of $t_{n+1}$ not intersecting $a_{n}$. In this case, we call $a_{n+1}$ the \textit{parallel side} of $a_{n}$ in $t_{n+1}$.
    \item If $t_{n+1}$ is obtained from $t_n$ as one of the three outer triangles of the medial subdivision (as seen on the right diagram), then two of its sides are contained in $t_n$. Those sides inherit their letter from the associated side in $t_n$, e.g. the side of $t_{n+1}$ contained in the side labelled $a_{n}$ of $t_n$ is named $a_{n+1}$. The remaining side of $t_n$ is to be named according to the convention of case (a).
\end{enumerate}

As for angles, $\alpha_n$ (resp. $\beta_n, \gamma_n)$ will denote the angle opposite $a_n$ (resp. $b_n, c_n$). 
\begin{figure}[H]
	\begin{center}     
    	\includegraphics[width=1\textwidth]{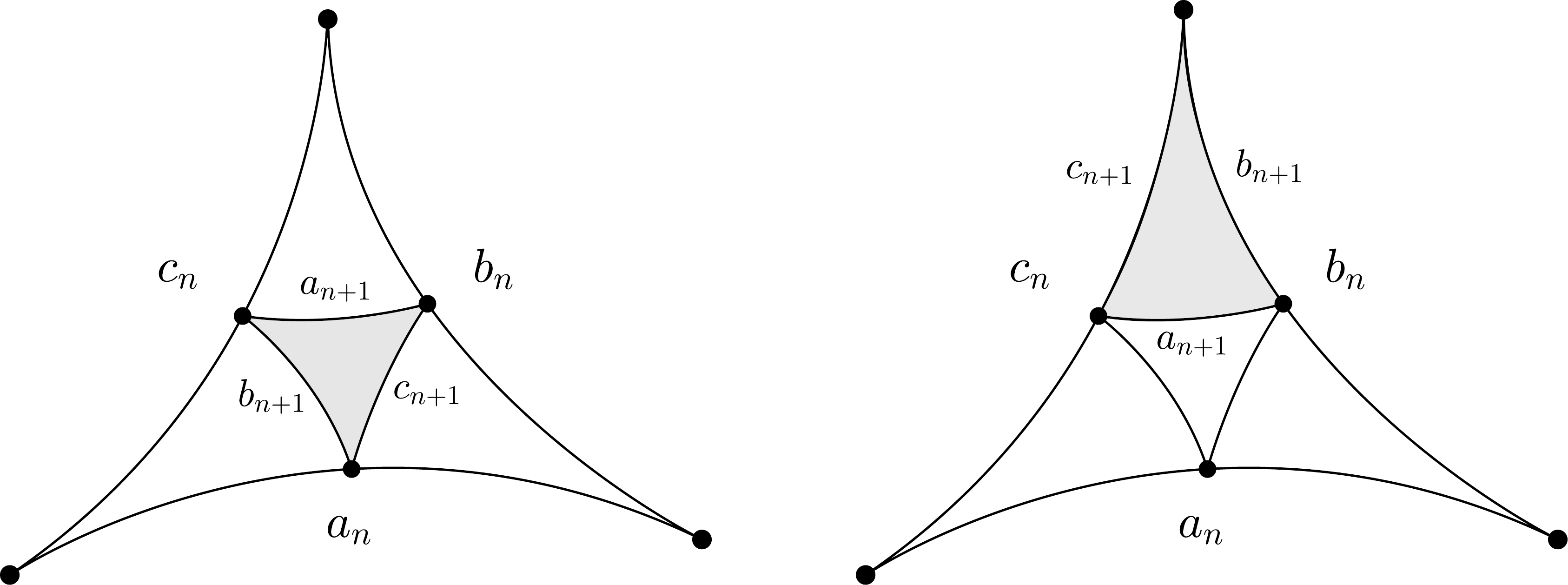}
        	\caption{: Case (a) on the left and (b) on the right. In both cases $t_{n}$ is the outer triangle and the nested triangle $t_{n+1}$ is highlighted in grey.}
        	\label{fig:sequence}
	\end{center}
\end{figure}

\begin{main}
\label{thm:angles}
For any sequence of nested triangles $t_0,t_1,\ldots$ and for all $n\in\mathbb N$, there exists $l_\alpha,L_\alpha>0$ such that:
\[ \alpha_0 \cdot l_{\alpha} < \alpha_n < \alpha_0\cdot L_\alpha \]
In addition, $l_\alpha$ (resp. $L_\alpha$) approaches $1$ from below (resp. above) as all the side lengths of $t_0$ become smaller. 
\end{main}

This theorem itself will be obtained as a consequence of Theorem C and a similar statement regarding the heights of triangles in our triangulations (Proposition \ref{lemma:altitude}).

\begin{main}
\label{thm:lengths}
For any sequence of nested triangles $t_0,t_1,\ldots$ and for all $n\in\mathbb N$, there exists $l_a, L_a>0$ such that:
\[a_0\cdot l_a \leq  2^n \cdot a_n \leq a_0 \tag{\textit{Hyperbolic geometry}}\]
\[a_0 \leq 2^n \cdot a_n \leq a_0\cdot L_a \tag{\textit{Spherical geometry}}\]
In addition, in the non-trivial cases where there exists at least some integer $n\in\mathbb N$ such that $a_{n+1}$ is obtained as the parallel side of $a_n$ in $t_{n+1}$, the inequalities are strict and $l_a$ (resp. $L_a$) approaches $1$ from below (resp. above) in the hyperbolic (resp. spherical) setting as all the side lengths of $t_0$ become smaller. 
\end{main}

In addition to bringing light on a very natural and elementary object, the study of the medial triangle subdivision in the non-Euclidean setting is motivated by our ongoing work on \textit{acute triangulations} of \textit{Riemannian triangle complexes}. A Riemannian triangle complex is a 2-dimensional simplicial complex in which each simplex is given its own individual Riemannian metric. While acute triangulations have been extensively studied in the Euclidean planar setting, the only existing result in the Riemannian setting is a highly non-constructive existence result for 2-dimensional Riemannian manifolds (\cite{Verdiere}). Even in the constant curvature setting, it is currently unknown whether complexes of spherical or hyperbolic polygons can be acutely triangulated or not. In a follow-up article, we will lay out a constructive existence result for the class of spherical and hyperbolic triangle complexes with finite isometry types. Our methods will exploit Theorem \ref{thm:main} to explicitly transport a new particular Euclidean acute triangulation scheme due to Chris Bishop (\cite{BS}) onto a fine enough medial triangle subdivision of the triangle complex to control the angle distortion occurring during the transport (Fig. \ref{fig:Riemannian} illustrates this heuristic for a hyperbolic triangle complex). 

\begin{figure}[H]
	\begin{center}     
    	\includegraphics[width=\textwidth]{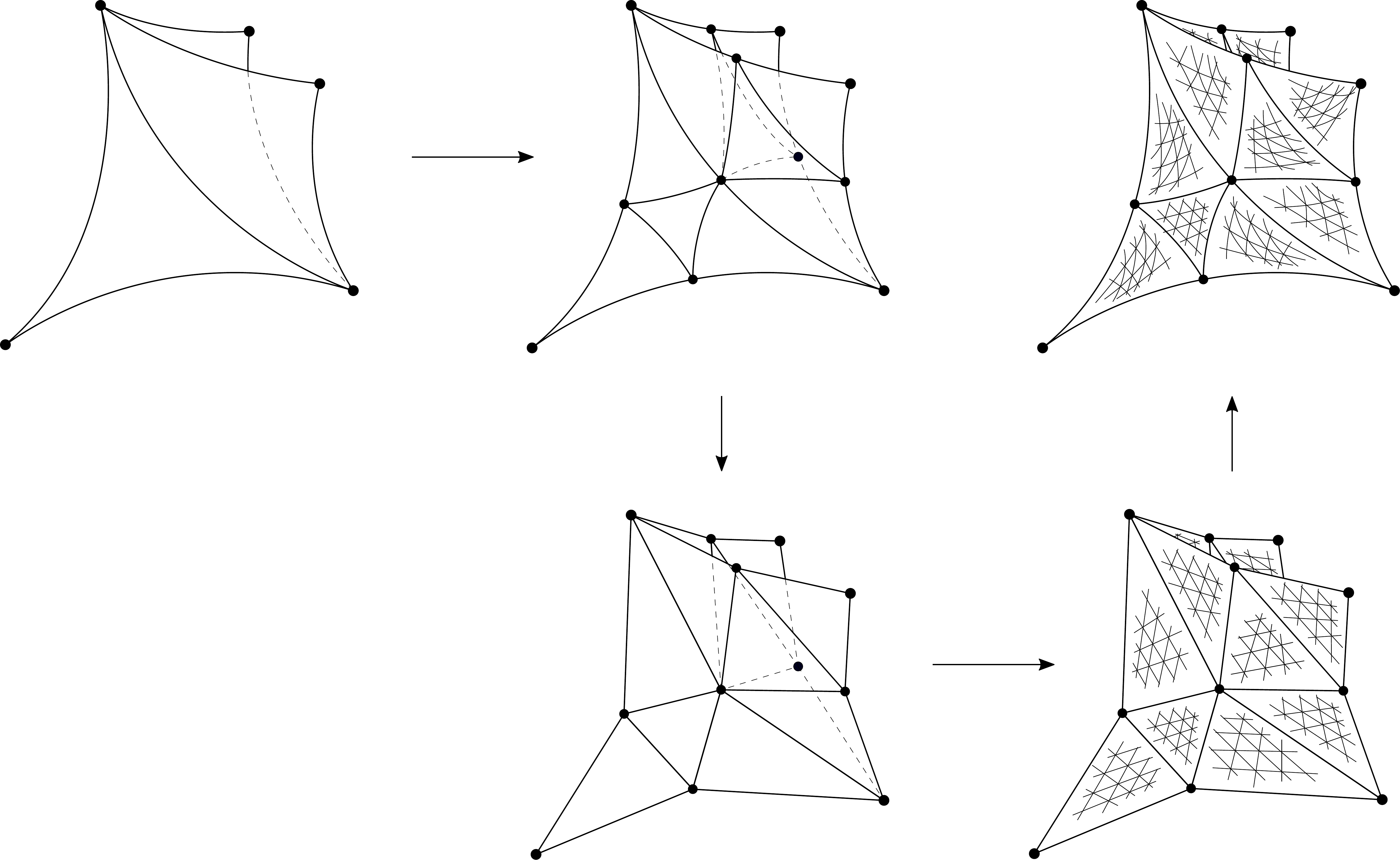}
        	\caption{}
        	\label{fig:Riemannian}
	\end{center}
\end{figure}

Before we proceed with the proof, we provide examples to convince the reader of the unusual behaviours exhibited by the medial triangle subdivision in the non-Euclidean cases. Indeed, in the Euclidean case, this subdivision yields $4$ congruent  triangles that are obtained from the original one by a similarity with scale factor one half. Therefore each iteration preserves the angles and halves the lengths, making our results trivial remarks. However, in the presence of non-zero curvature, the situation is not so straightforward. In fact, as we set off to showcase, a surprising fact about this subdivision in curved spaces is that a single step of the subdivision can be ``arbitrarily degenerate''. This should make our result perhaps slightly more surprising to the reader. 

\section{Examples}
\label{sec2}

\noindent\textbf{Example 1.} We first show how to construct a family of triangles for which the ratio $\alpha_1/\alpha_0$ is unbounded (for one of the 4 possible choices of $\alpha_1$). More strikingly,  $\alpha_0$ can be taken arbitrarily small and $\alpha_1$  arbitrarily close to $\pi$. Consider the following example of an isosceles triangle $t_0$ in the hyperbolic plane (seen in the Poincar\'e disk model on Figure \ref{fig:thong}). We label its vertices $A,B,C$ and its midpoints $D,E,F$, as indicated on the figure. In the Euclidean case, the angle $\alpha_0$ and $\alpha_1$ would be equal. Here however, fixing $A$, we can extend the geodesic segment $BC$ to a line and have the points $B$ and $C$ move further apart from each other at equal speed on this line. Doing so will drag the midpoints $F$ and $E$ arbitrarily close to $B$ and $C$ (in the Euclidean metric), making the angle $\alpha_1$ arbitrarily close to $\pi$. This construction is valid for any choice of $A$, and in particular we can choose $A$ to be arbitrarily close to the boundary, making the angle $\alpha_0$ arbitrarily close to $0$.

\begin{figure}[H]
	\begin{center}     
    	\includegraphics[width=0.55\textwidth]{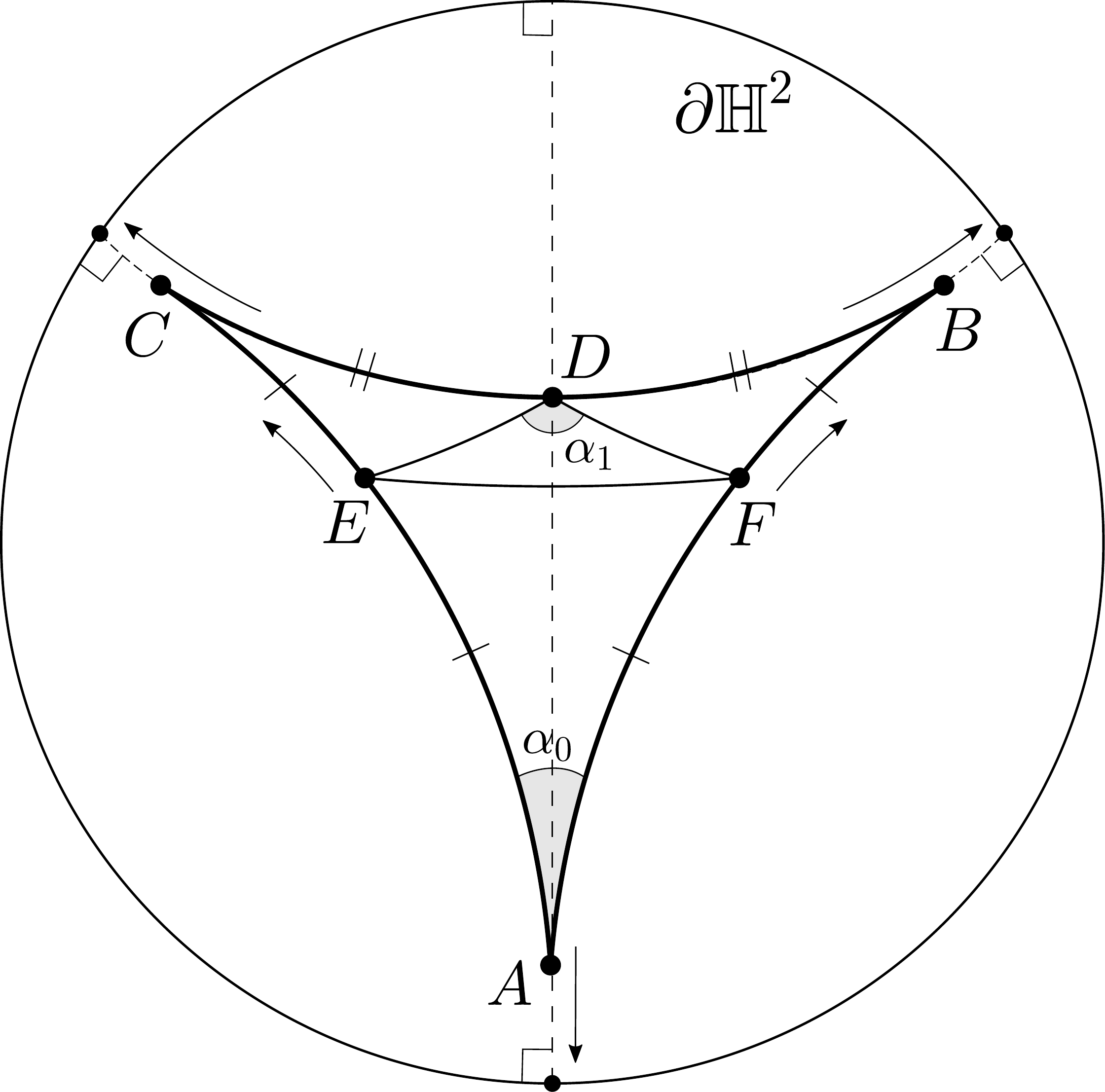}
        	\caption{}
        	\label{fig:thong}
	\end{center}
\end{figure}

\noindent\textbf{Example 2.} The second surprising phenomenon about the medial triangle subdivision in non-Euclidean geometries is that its behaviour with respect to the angles depends on the initial triangle, namely: in certain triangles it will increase the corresponding angles while it will decrease it in others. To showcase this behaviour, we will provide a quantitative geometric criterion for isosceles triangles to exhibit either one behaviour or the other. Consider an isosceles triangle with vertices $A,B,C$ and sides of length $2a$ and $2b$ (see Figure \ref{fig:isosceles}). Let $D,E,F$ be the midpoints of the sides $BC, CA$ and $AB$. Let $Q$ be the point of intersection of the geodesic segments $FE$ and $AD$; let $\alpha,\beta,\alpha',\beta'$ be the angles described on Figure \ref{fig:isosceles}. 

By symmetry, the triangle $DCA$ has a right angle at $D$. For the same reason, $DEQ$ has a right angle at $Q$. Let $u$ denote the length of the geodesic segment $DE$ and $h$ that of the segment $DA$. Hyperbolic trigonometry identities (\cite{WT}, p81) in the right-angled triangle $DEQ$ give us:

\[
\cot \beta' \> \cot \alpha' = \cosh u
\label{eq:Pytha}\tag{1}
\]

Likewise, the hyperbolic sine rule in the triangles $DCE$ and $DEA$ give us the two identities:

\[
\frac{\sin \beta}{\sinh u} = \frac{\sin (\frac{\pi}{2}-\alpha')}{\sinh a} = \frac{\cos \alpha'}{\sinh a}
\]

\[
\frac{\sin \alpha}{\sinh u} = \frac{\sin \alpha'}{\sinh a}
\]

which, combined, give us:

\[
\tan \alpha' = \frac{\sin \alpha}{\sin \beta}
\label{eq:Sine}\tag{2}
\]

\vspace{-2em}
\begin{figure}[H]
	\begin{center}     
    	\includegraphics[width=0.35\textwidth]{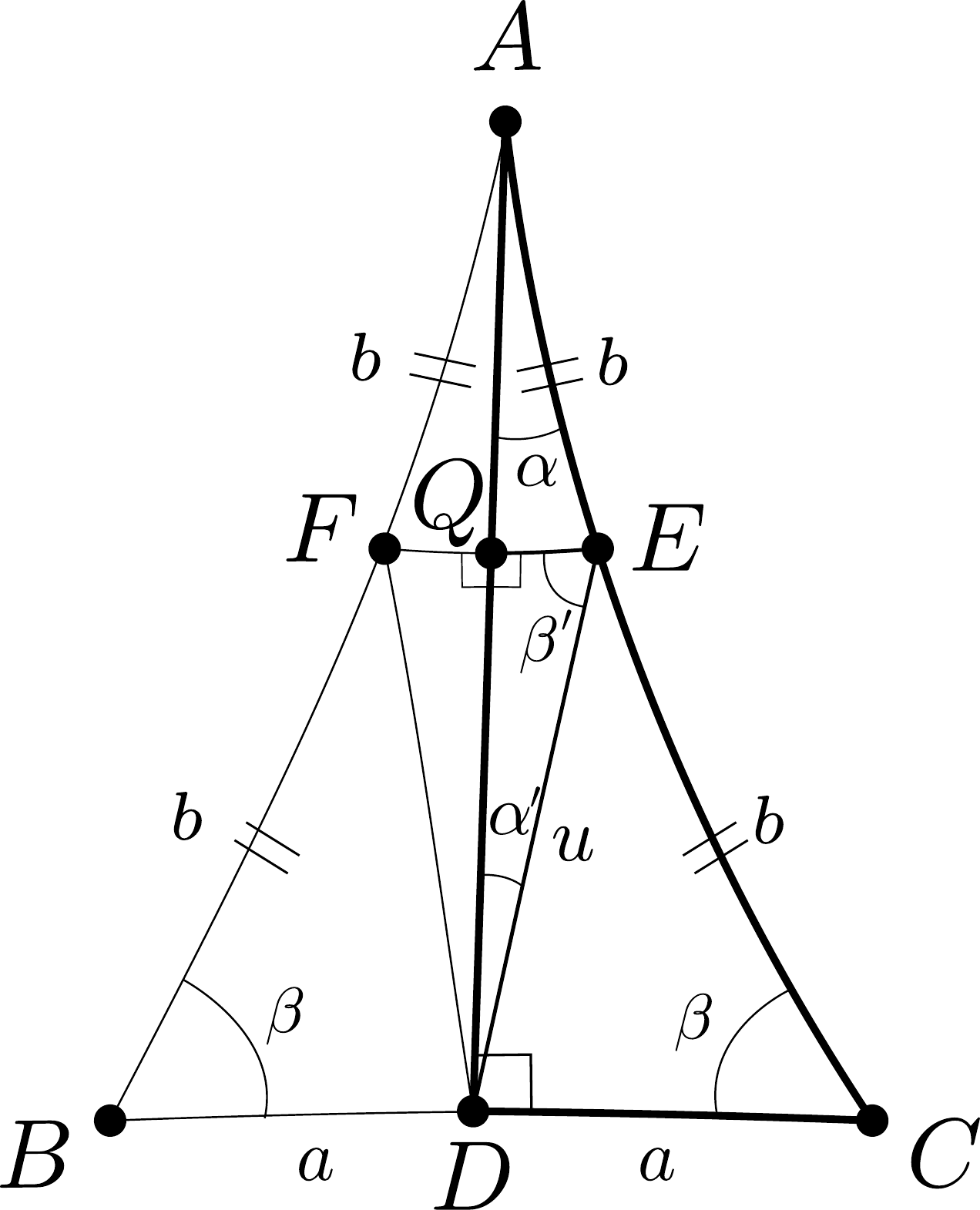}
        	\caption{}
        	\label{fig:isosceles}
	\end{center}
\end{figure}
\vspace{-2em}
Finally, the dual hyperbolic cosine rule (\cite{WT}) applied to the triangle $DBC$ informs us that:

\[
\cos \beta = \sin \alpha \> \sin \frac{\pi}{2}\> \cosh h - \cos \frac{\pi}{2}\>\cos \alpha = \sin \alpha \> \cosh h
\label{eq:Cosine}\tag{3}
\]

We are looking for a condition on the isosceles triangle $ABC$ to ensure that either $\beta' > \beta$ or its converse is true, or equivalently that $\tan \beta' > \tan \beta$ or otherwise. Combining equations (\ref{eq:Pytha}), (\ref{eq:Sine}) and (\ref{eq:Cosine}), we obtain:

\[
\tan \beta' > \tan \beta \iff \left(\frac{\sin \alpha}{\sin \beta} \> \cosh u\right)^{-1} > \frac{\sin \beta}{\cos \beta} \iff \cosh h > \cosh u
\]

There are thus two distinct and opposite scenarios possible, in which either:

\begin{itemize}
    \item $u > h$, in which case the medial subdivision produces a smaller angle, i.e. $\beta'<\beta$ (left diagram of Figure \ref{fig:pathology} shows an example in the Poincar\'e disk model).
    \item  $u < h$, in which case the medial subdivision produces a larger angle, i.e. $\beta'>\beta$  (right diagram of Figure \ref{fig:pathology}).
\end{itemize}
\vspace{-2em}
\begin{figure}[H]
	\begin{center}     
    	\includegraphics[width=\textwidth]{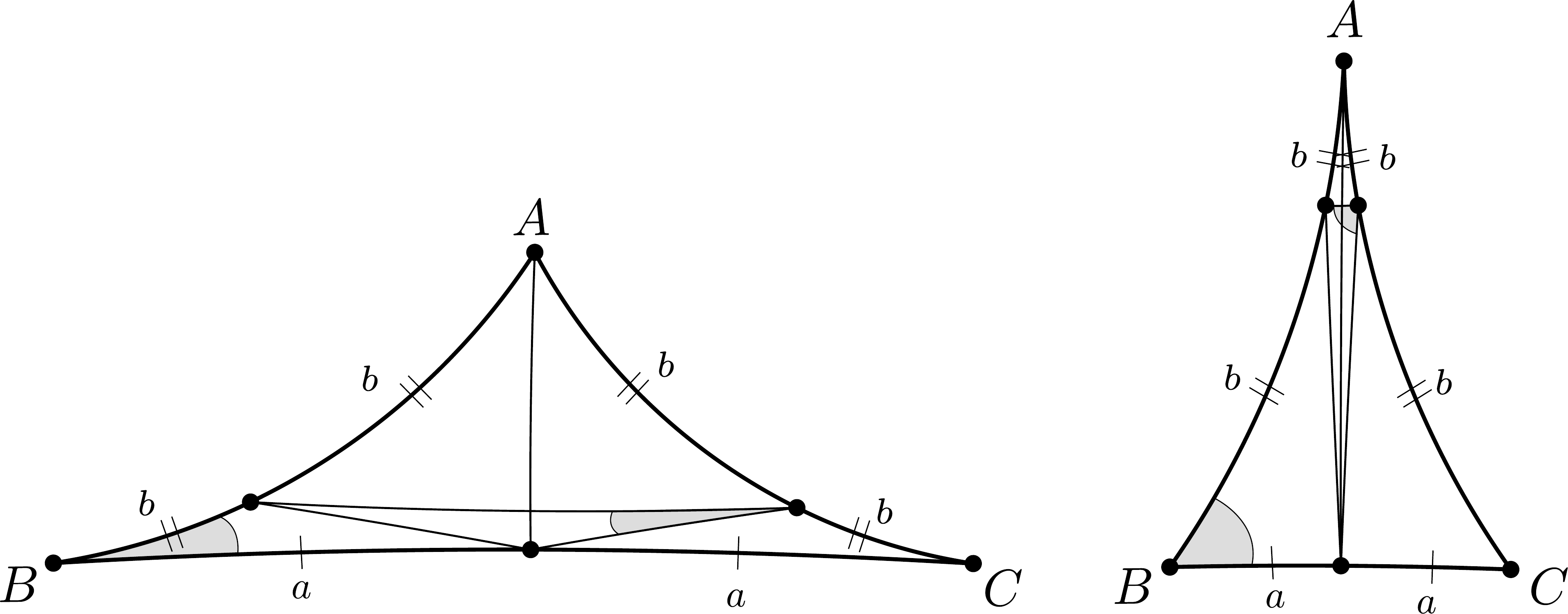}
        	\caption{}
        	\label{fig:pathology}
	\end{center}
\end{figure}

\begin{rem}
The exact same reasoning works transposed in the spherical setting but yields the opposite inequalities and behaviours, namely the two cases on Fig. \ref{fig:pathology} are reversed. This is because equalities (\ref{eq:Sine}),(\ref{eq:Cosine}) and (\ref{eq:Pytha}) all remain the same with regular sines and cosines instead of their hyperbolic counterpart. The two scenarios previously highlighted are however to be swapped, since the cosine function is decreasing on the interval $[0,\pi]$, while the hyperbolic cosine function is increasing on that same interval.
\end{rem}

To briefly address how the subdivision behaves with respect to lengths, we provide two particularly striking examples in the spherical case. 

\noindent\textbf{Example 3.} We first show that, for certain spherical triangles, lengths can be arbitrarily distorted by the medial subdivision. Consider a spherical isosceles triangle $ABC$, with right angle at $A$ and equal sides $|AB|=|AC|$. As $B$ and $C$ approach the antipodal point of $A$, $|BC|$ becomes arbitrarily small, while $|EF|$ will approach a quarter of the equatorial circle between those two poles. In that sense, the spherical upper bound of Theorem B should perhaps appear less natural, as we can create triangles in which one of the sides will have its corresponding side in the next step of the medial triangle subdivision arbitrarily larger. 

\begin{figure}[H]
	\begin{center}     
    	\includegraphics[width=0.45\textwidth]{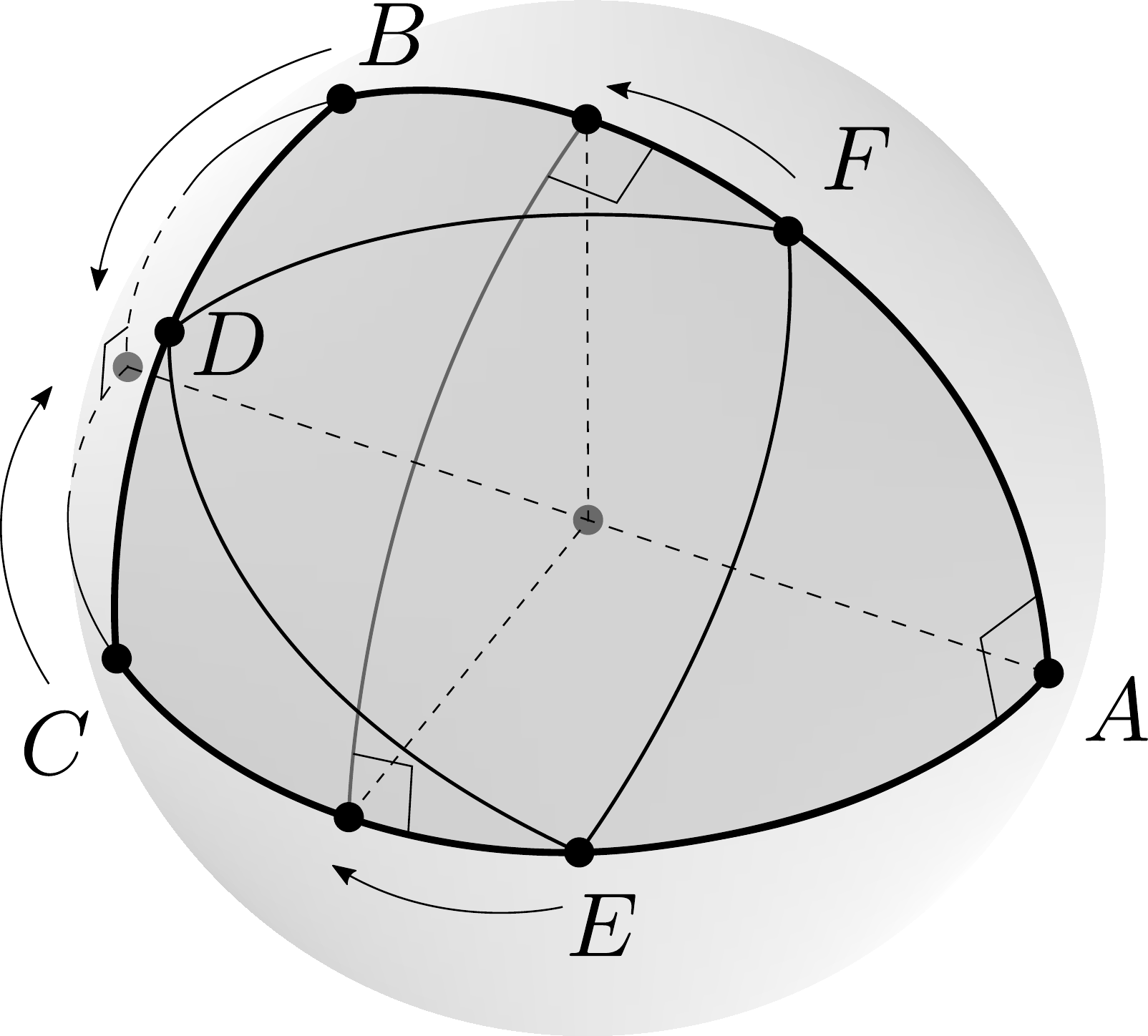}
        	\caption{}
        	\label{fig:bong}
	\end{center}
\end{figure}
\vspace{-1em}

\noindent\textbf{Example 4.} Lastly, we provide an example to show that, for any $\epsilon>0$ and for any given $N\in \mathbb N$, we can find a triangle $T_0^\epsilon$ for which every edge of $T_N^\epsilon$ lies in the $\epsilon$-neighbourhood of the union of the sides of $T_0^\epsilon$. Moreover, $\text{diam}(T_N^\epsilon)>\frac{\pi}{2}$. We first remark, that, while our subdivision is not defined for a triangle on the equator circle, one can nevertheless imagine what the subdivision would resemble in the case where the three vertices $A$, $B$ and $C$ are equidistributed on the equator, as there is still a unique geodesic between all midpoints in this case. It is easy to see that taking the midpoints of this triangle gives another triple of points of the equator which are also equidistributed. By continuity, if we consider a triangle $T_0^\epsilon$ whose vertices are all equidistributed on a latitude circle close to the equator and let its vertices approach $A$, $B$ and $C$ respectively, we see that its midpoints will also stay close to the midpoints of $ABC$. By induction, for any finite number $N$ of subdivisions and any $\epsilon>0$, choosing $T_0^\epsilon$ to lie on a latitude circle sufficiently close to the equator then guarantees that all edges of $T_N^\epsilon$ remain within the $\epsilon$-neighbourhood of the union of the sides of $T_0^\epsilon$. Because of this, it is clear that for $\epsilon$ sufficiently small, $\text{diam}(T_N^\epsilon) > \frac{\pi}{2}$.

\begin{figure}[H]
	\begin{center}     
    	\includegraphics[width=0.5\textwidth]{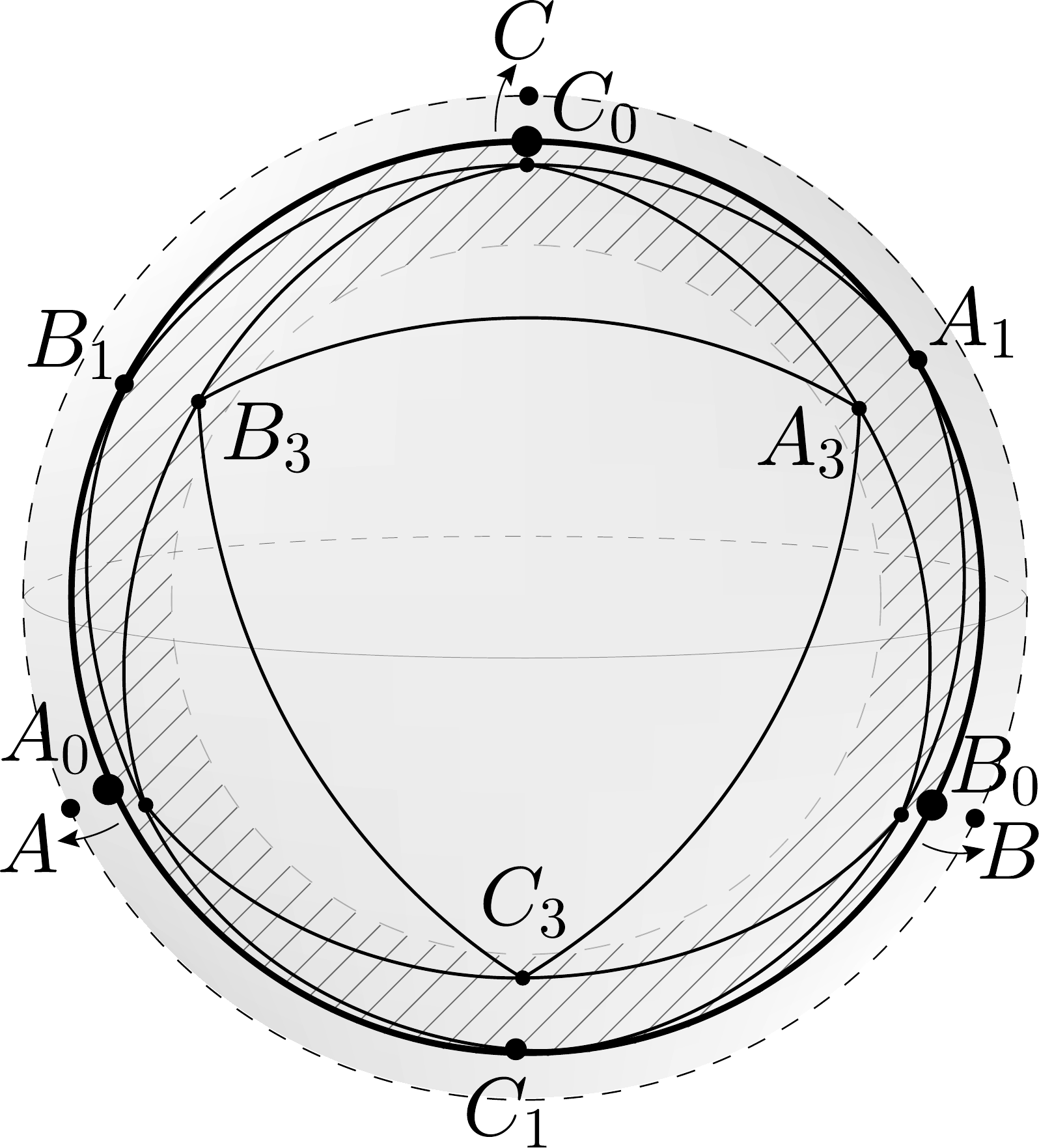}
        	\caption{: This figure illustrates Example 4, with the hatched annulus corresponding to the $\epsilon$-neighbourhood of the union of the sides of $T_0$ and $N=3$. Smaller values of $\epsilon$ and larger values of $N$ are achieved by choosing the vertices of $T_0$ closer to those of $ABC$.}
        	\label{fig:equi-sphere}
	\end{center}
\end{figure}
\vspace{0em}

\section{Stabilisation of Lengths}
\label{sec3}

Despite the various unusual behaviours showcased by our previous examples, we claim that as $n$ grows, the refining triangulations eventually ``stabilise'' to the limiting Euclidean case, in the sense of Theorem B and Theorem C. In this section, we establish our notation and focus on the behaviour of the edge-lengths of the subdivision. The core of this section is our proof of Theorem C, which gives a precise sense to the ``stabilisation of lengths'' observed in the medial triangle subdivision. Theorem C will also play a crucial part in our proofs of Proposition \ref{lemma:altitude} and Theorems A and B.

While the study of the behaviour of the heights in the subdivision is delayed to Section 4, the constructions used towards the proof of Theorem C rely heavily on taking orthogonal projections and measuring heights. We thus begin by making clear the meaning of ``height'' in the positive curvature setting. Indeed, while there is a unique orthogonal projection from any point to any line in the non-positive curvature setting, the situation is slightly more subtle in the spherical case. In the spherical setting, if we fix a point $p$ and a great circle $\mathcal C$, there are two possible cases, depending on whether $p$ is a pole of the sphere for $\mathcal C$ considered to be the equator circle. If it is not (left diagram of Fig. \ref{fig:notation-spherique}), there is a unique geodesic arc which realises the distance of $p$ to $\mathcal C$. This arc is what we refer to as the \textit{altitude} drawn from $p$ onto $\mathcal C$ and we call its length the \textit{height} of $p$ to $\mathcal C$. The point of intersection between this arc and $\mathcal C$ is referred to as the \textit{orthogonal projection} of $p$ onto $\mathcal C$. Note that, in this case, the height is always strictly less than $\frac{\pi}{2}$. If $p$ is a pole of the sphere for $\mathcal C$ considered to be the equator circle (right diagram of Fig. \ref{fig:notation-spherique}), the height of $p$ to $\mathcal C$ is defined to be equal to $\frac{\pi}{2}$, while both the altitude from $p$ to $\mathcal C$ and its orthogonal projection on $\mathcal C$ are undefined. Notice that in both cases, the height is the minimal distance from $p$ to any point on the line $\mathcal C$. Fortunately, all our proofs relying on altitudes and orthogonal projections will take place in a setting where the distance between any two points is strictly less than $\frac{\pi}{2}$, therefore ensuring the second case where $p$ is a pole with respect to the equator circle $\mathcal C$ cannot happen.

\begin{figure}[H]
	\begin{center}     
    	\includegraphics[width=0.65\textwidth]{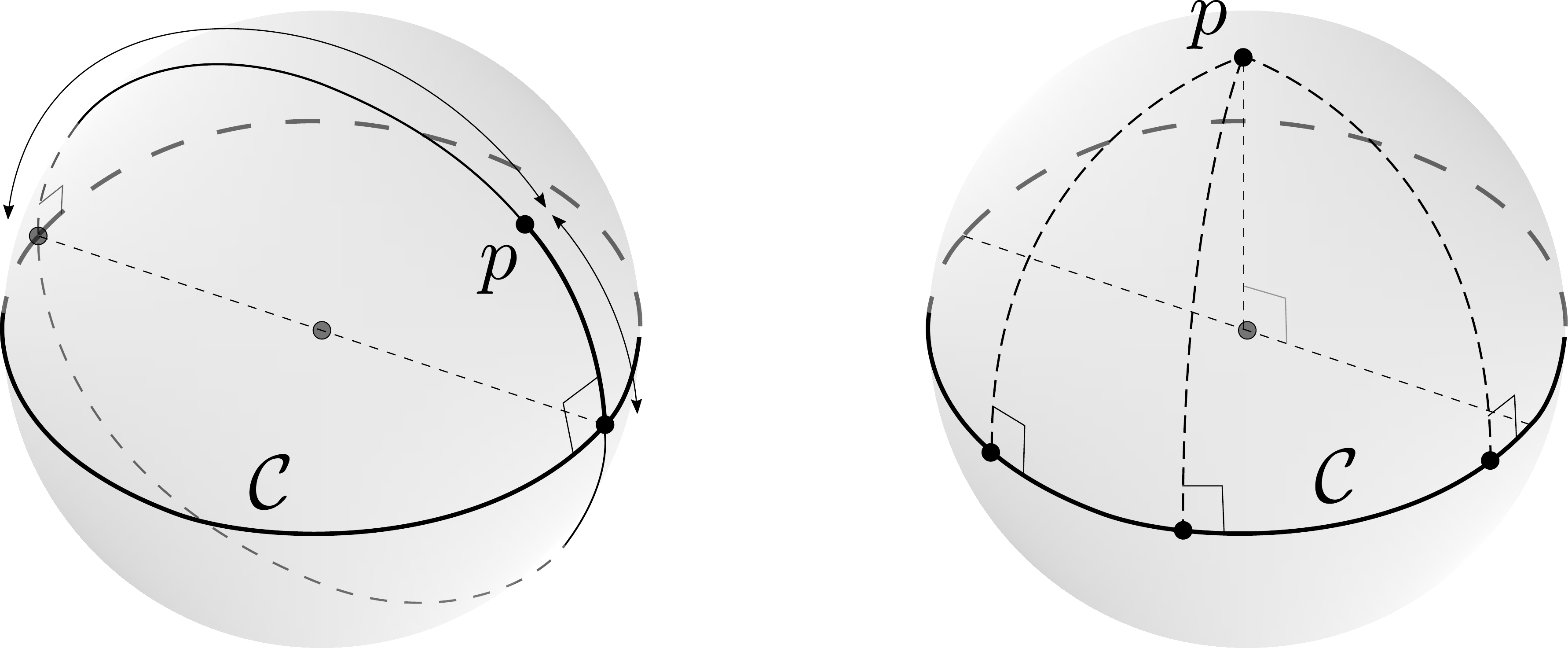}
        	\caption{: A point on the sphere has either two (left diagram) or uncountably many (right diagram) orthogonal projections to a line.}
        	\label{fig:notation-spherique}
	\end{center}
\end{figure}

We begin by stating the following lemma:

\begin{lemma}
\label{lemma:mid-edge}
For any sequence of nested triangles $t_0,t_1,\ldots$, and for all $n\in\mathbb N$, the following inequalities hold:
\[ a_{n+1} \leq \frac{a_n}{2}\tag{\textit{Hyperbolic geometry}}\]
\[a_{n+1} \geq\frac{a_n}{2}\tag{\textit{Spherical geometry}}\]
(and similarly for $b_{n+1} $ and $c_{n+1}$). The inequalities are strict when $a_{n+1}$ is obtained as the parallel side of $a_{n}$ in $t_{n+1}$.
\end{lemma}

\begin{proof}
The two cases where $a_{n+1}$ is contained in $a_n$ are trivial. The only case of interest is thus when $a_{n+1}$ is obtained as the parallel side of $a_{n}$ in $t_{n+1}$. In that case, it is a consequence of the observation that the hyperbolic plane (resp. the Euclidean plane) is a CAT(-1) space (resp. a CAT(0) space), and thus also CAT(0) space (resp. a CAT(1) space) (\cite[II.1.13]{BH}). The lemma follows directly from the CAT($\kappa$) inequality. 
\end{proof}
While this observation is a well known fact, a later construction of ours will provide an elegant alternative proof of Lemma \ref{lemma:mid-edge} later on in this section, see Remark \ref{newproof} in our proof of Theorem \ref{thm:lengths}. 

\begin{lemma}
\label{lemma:lengths}
In the spherical case, there exists an integer $N$ and a positive constant $C<2$ depending only on $t_0$, such that, for all $n>N$, we have $a_{n+1} \leq C \> \frac{a_n}{2}$. 
\end{lemma}

We now note that Lemma \ref{lemma:lengths} clearly holds in the hyperbolic case because of Lemma \ref{lemma:mid-edge}, choosing $C=1$ and $N=0$ in the statement. But we also wish to establish an upper bound in the spherical case. For that purpose, we will think of a geodesic triangle $T$ on the sphere as a Jordan curve and define its \textit{interior} as the connected component of $\mathbb S^2-T$ that is contained in the open hemisphere containing $T$. 

\begin{lemma}
\label{claim0}
Given a geodesic $c:[0,l]\rightarrow \mathbb S^2$ joining two points $u=c(0)$ and $v=c(l)$ on a geodesic triangle $T$, the restriction of $c$ to the open interval $(0,l)$ lies in the interior of $T$.
\end{lemma} 

\begin{proof} To fix the notation, $PQR$ will denote a geodesic triangle and $X,Y,Z$ the midpoints of its three edges $QR$, $RP$ and $PQ$. Note that we have that $d_{\mathbb S^2(X,Y)}<\pi$ and likewise for the other two pairs. There is then a unique minimal arc joining joining $X$ and $Y$ and this arc is the intersection of $\mathbb S^2$ with the positive cone in $\mathbb E^3$ spanned by $X$ and $Y$, seen as unit vectors in $\mathbb E^3$. Thus all the points of this geodesic arc are of the form $x X + y Y$, with $x,y \geq 0$ and $x+y > 1$. Since $X$ is the midpoint of $QR$, it can be expressed as $\lambda (Q+R)$, with $\lambda > \frac{1}{2}$. Likewise $Y$ can be expressed as $\mu (P+R)$, $\mu > \frac{1}{2}$. This shows that each point on the geodesic segment joining $X$ and $Y$ can be written as a sum $x\lambda (P+Q)+y \mu (Q+R)=\alpha P + \beta Q + \gamma R$ with $\alpha+\beta+\gamma> 1$ and $\alpha, \beta, \gamma \geq 0$. The entire geodesic segment thus lies in the intersection between the positive cone in $\mathbb E^3$ spanned by $P, Q$ and $R$ and $\mathbb S^2$ and is thus contained within the triangle $PQR$. For any points of the geodesic distinct from $X$ and $Y$, we have $\alpha,\beta,\gamma >0$, which shows that these points lie in the interior of the triangle $PQR$.
\end{proof}

We shall now need the following important lemma to prove Lemma \ref{lemma:lengths}:

\begin{lemma}
\label{claim1}
For all $\epsilon > 0$, there exists $N\in \mathbb N$ such that all the edge-lengths of $t_n$, for $n>N$, are smaller than $\epsilon$.
\end{lemma}
\indent We point out here that this is still a weaker statement than that of Lemma \ref{lemma:lengths}, which tells us that any sequence of edge lengths $(a_n)_{n\in \mathbb N}$ not only converges to $0$ but is also bounded above by a geometric sequence. \\
\begin{proof}
Given a sequence of nested triangles $t_0, t_1, \ldots$, we define the sequences of points $(A_n)_{n\in\mathbb N}$, $(B_n)_{n\in\mathbb N}$ and $(C_n)_{n\in\mathbb N}$ consisting, for each $n$, of the vertices of $t_n$ incident to the angles $\alpha_n$, $\beta_n$ and $\gamma_n$ respectively. In our setting, $t_0$ is a closed compact subset of the open hemisphere, thus each of the three sequences $(A_n)$, $(B_n)$ and $(C_n)$ have subsequences $(A_{n_k})$, $(B_{n_k})$ and $(C_{n_k})$ converging to $A$, $B$ and $C$ respectively, with all three points lying in $t_0$ (by Lemma \ref{claim0}).

\begin{figure}[H]
	\begin{center}     
    	\includegraphics[width=0.7\textwidth]{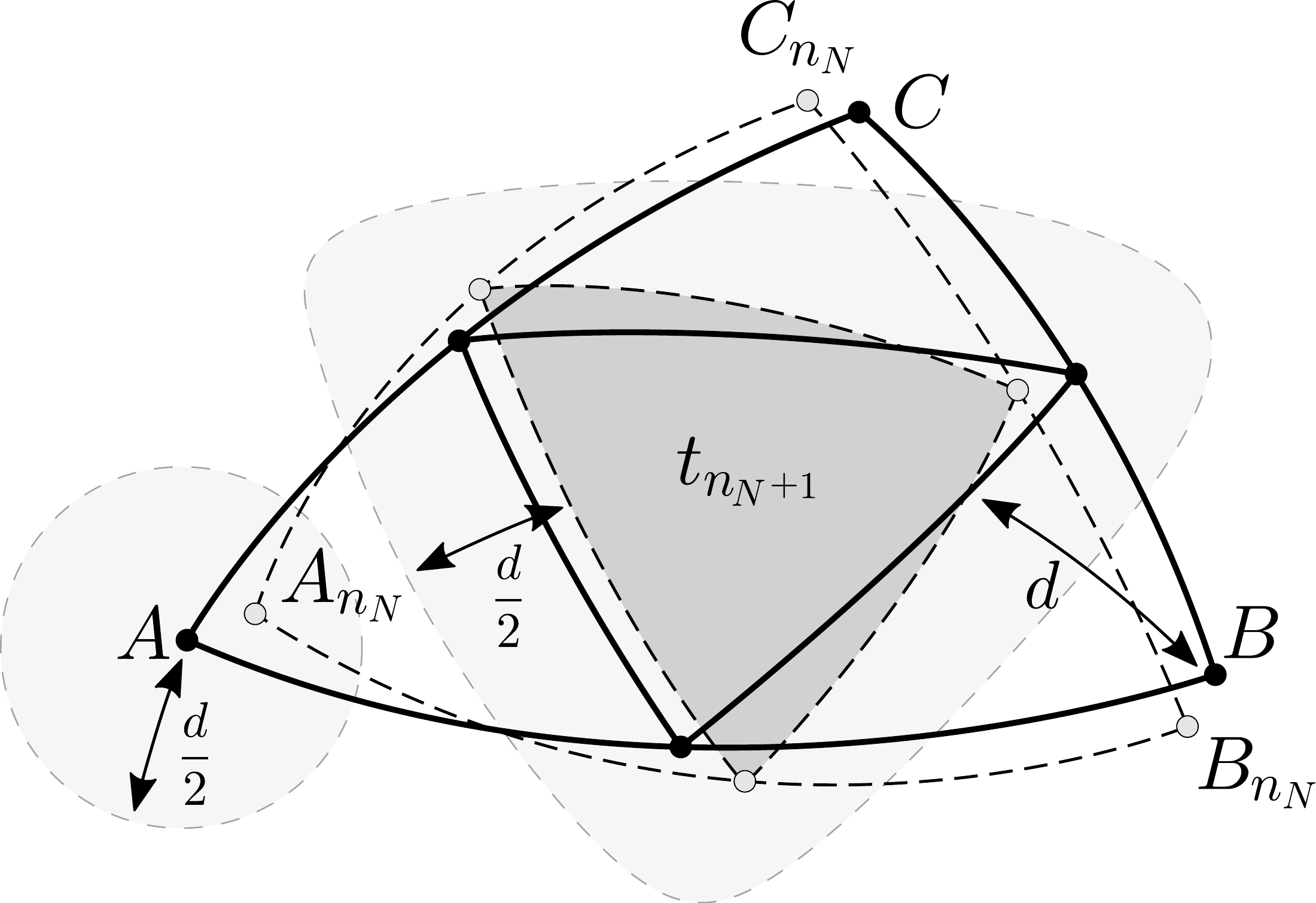}
        	\caption{}
        	\label{fig:ABC}
	\end{center}
\end{figure}

Suppose now by contradiction that the sequences $(A_{n_k})$, $(B_{n_k})$ and $(C_{n_k})$ do not converge to the same point. There are two possible cases.

\textit{Case 1:} $A\neq B \neq C$. Lemma \ref{claim0} guarantees that the vertices of $T=ABC$ and the innermost triangle $T'$ of $T$ in the medial triangle subdivision form two disjoint closed sets. Denote then by $d$ the minimum of the three distances between $T'$ and each of the points $A$, $B$ and $C$. Then the $d\slash 2$-neighbourhoods of $A$ and $T'$ are disjoint (likewise for $B$ and $C$). Since geodesics between any two points of the open hemisphere are unique and continuous with respect to their endpoints, we have that there exists $N$ such that the $d\slash 2$-neighbourhood of each edge in the medial triangle subdivision of $T$ contains the corresponding edge of $A_{n_N} B_{n_N} C_{n_N}$(see Fig. \ref{fig:ABC}).  Assuming that $t_{n_{N}+1}$ is the innermost triangle of $t_{n_N}$, we can then guarantee that $A$ and $t_{n_{(N+1)}}$ are disjoint. But this is impossible as $A$ lies in $t_{n_{(N+1)}}$. If instead $t_{n_{N}+1}$ were not the innermost triangle but (for example) the triangle $B_{{n_N}+1} A_{n_N} C_{{n_N}+1}$ (the bottom-left triangle on Fig. 10), then we could now guarantee that $B$ (or $C$) and $t_{n_{(N+1)}}$ are disjoint. This is again impossible as $B$ (and $C$) lies in $t_{n_{(N+1)}}$. The other two cases are dealt with in the exact same fashion. 

\begin{figure}[H]
	\begin{center}     
    	\includegraphics[width=0.85\textwidth]{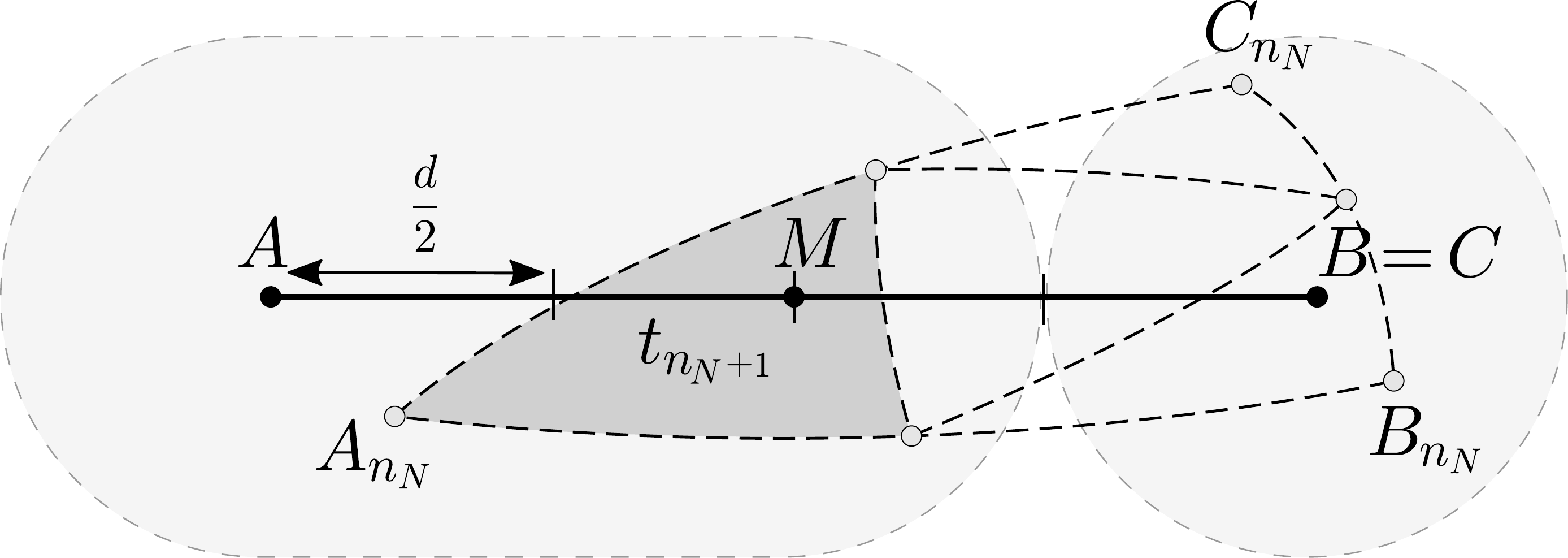}
        	\caption{}
        	\label{fig:ABC-2}
	\end{center}
\end{figure}

\textit{Case 2:} $A\neq B$ and $B=C$ (the other two possible cases are symmetric up to a relabelling of the vertices). We proceed with a similar argument as in the first case, replacing $T$ by the geodesic segment joining $A$ and $B$, $T'$ by the geodesic segment joining the midpoint $M$ of $AB$ to $B=C$ and letting $d=d_{\mathbb S^2}(A,B)\slash 2$. In this degenerate case, we consider the ``medial triangle subdivision'' of $T$ where the midpoints of $AB$ and $AC$ coincide with $M$, the midpoint of $BC$ coincides with $B$, and the tree edges joining the three midpoints of $BC$, $CA$ and $AB$ are the point $M$ and the edge $MB$ (counted twice).  We then similarly obtain an integer $N$ such that the $d\slash 2$-neighbourhood of each edge in the ``medial triangle subdivision'' of $T$ contains the corresponding edge of $A_{n_N} B_{n_N} C_{n_N}$. For the three possible choices of $t_{n_{N}+1}$ that do not contain $A_{n_N}$, we can consider the $d\slash 2$-neighbourhoods of $T'$ and $A$ to guarantee that $A$ and $t_{n_{(N+1)}}$ are disjoint. But this is impossible since $A$ lies in $t_{n_{(N+1)}}$. If instead $t_{n_{N}+1}$ was the triangle containing $A_{n_N}$ (see Fig. \ref{fig:ABC-2}), then we could now consider the $d\slash 2$-neighbourhoods of $AM$ and $B=C$ to guarantee that $B=C$ and $t_{n_{(N+1)}}$ are disjoint. This is again impossible since $B=C$ lie in $t_{n_{(N+1)}}$. This concludes the proof of Case 2.
\end{proof}

\begin{rem}
\label{small}
Owing to Lemmas \ref{claim0} and \ref{claim1}, we now know that, in the spherical setting, we can assume our triangles to be small enough to lie inside an open ball of radius $\frac{\pi}{4}$, so that the distance between any two points is strictly less than $\frac{\pi}{2}$. This guarantees that the height from any point lying on such a triangle to any line (great circle) intersecting the triangle must be strictly less than $\frac{\pi}{2}$. Indeed, as we noted before, the height from a point to a line is the minimal distance from this point to any point on the line. The second case of Figure \ref{fig:notation-spherique} will thus be safely averted from there on.  
\end{rem}

We now introduce some notation. For convenience, we write $a_n$ (resp. $b_n, c_n$) as $BC$ (resp. $CA$, $AB$) and the midpoints of $BC$, $CA$ and $AB$ by $D,E$ and $F$ (see Fig. \ref{fig:construction-both}). In the following proofs, we let $A', B',C',D'$ be the orthogonal projections of $A,B,C,D$ on the geodesic line (i.e. the great circle in the spherical setting) $FE$. We also consider the orthogonal projections $F'$ and $E'$ of $F$ and $E$ on the line $BC$. Note that these are all well-defined in light of Remark \ref{small}.

\begin{figure}[H]
	\begin{center}     
    	\includegraphics[width=\textwidth]{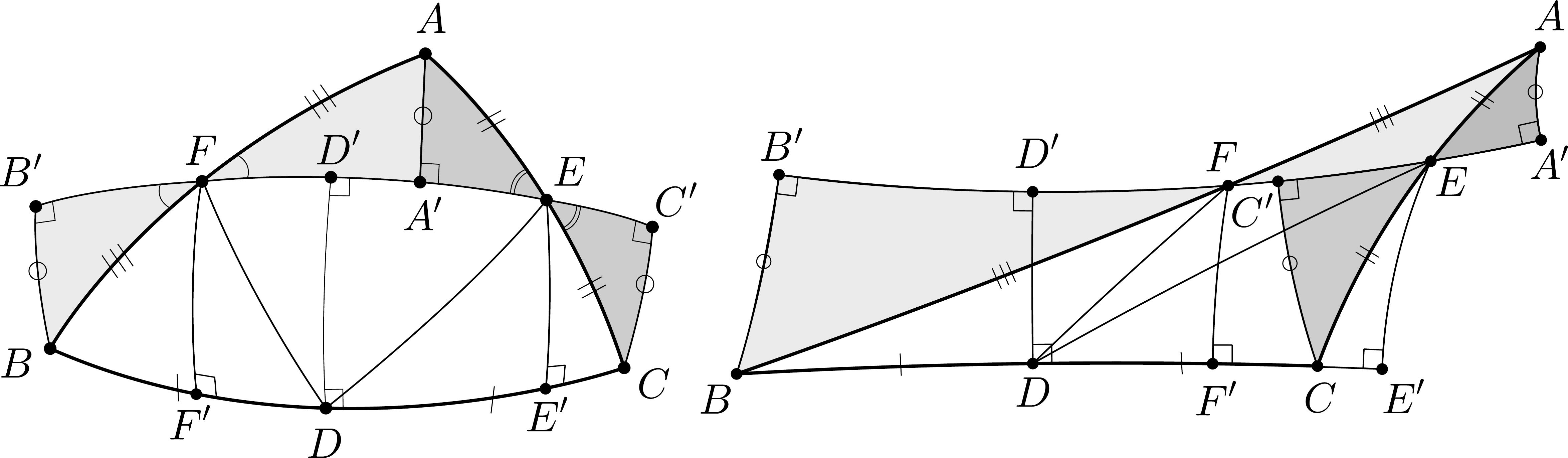}
        	\caption{: The construction for an acute spherical triangle (left) and an obtuse hyperbolic triangle (right).}
        	\label{fig:construction-both}
	\end{center}
\end{figure}

In the following discussion, we define a \textit{quadrilateral} $XYY'X'$ as the union of the geodesic segments $XY$, $YY'$, $Y'X'$ and $X'X$, provided that every pairwise intersection of the four interiors of these segments is empty. We remind the reader that a quadrilateral $XYY'X'$ for which the oriented angles $\angle Y'YX$ and $\angle YXX'$ are right angles (for the two possible orientations of the hyperbolic plane or the sphere) and the lengths of its sides $XX'$ and $YY'$ are equal is called a \textit{Saccheri quadrilateral} with \textit{base} $XY$ and \textit{summit\textit} $X'Y'$ (sometimes called a Saccheri isosceles birectangle). Saccheri quadrilaterals have a unique line of symmetry cutting both their base and summit sides perpendicularly through their midpoints (see for example \cite[\S 21]{Martin}). Likewise, a quadrilateral $XYY'X'$ in which the angles at $X,Y$ and $X'$ are right is called a \textit{Lambert} quadrilateral (sometimes called Lambert trirectangle) with \textit{apex} at $Y'$.

We first give a short proof of a trigonometric identity in Lambert quadrilaterals which will prove to be very useful in our proof of the main theorem and which we shall use throughout this article.

\begin{lemma}
\label{lemma:Lambert-Trigo}
In a hyperbolic Lambert quadrilateral $XYY'X'$ with apex $Y'$, the following identity holds:

\[\sinh |X'Y'| = \sinh |XY| \cdot \cosh |YY'| \]

The identity in spherical geometry is obtained by replacing hyperbolic trigonometric functions by spherical trigonometric functions:
\[\sin |X'Y'| = \sin |XY| \cdot \cos |YY'| \]
\end{lemma}
\vspace{-1em}
\begin{figure}[H]
	\begin{center}     
    	\includegraphics[width=0.4\textwidth]{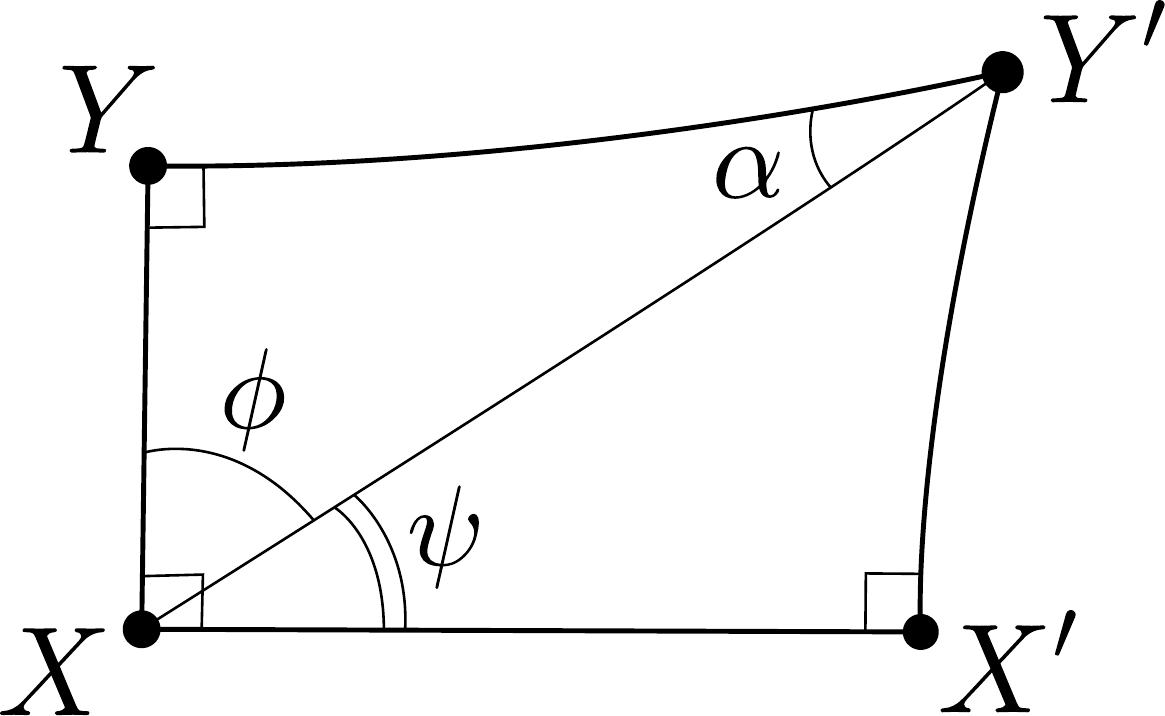}
        	\caption{}
        	\label{fig:lambert-proof}
	\end{center}
\end{figure}

\begin{proof}
 From the dual hyperbolic law of cosines in the right triangle $XY'Y$ (\cite[2.4.9]{WT}), we obtain $\cosh |YY'|=\frac{\cos \phi}{\sin \alpha}$. Morevoer, from the hyperbolic law of sines we obtain that $\frac{\sin \alpha}{\sinh |XY| }=\frac{1}{\sinh |XY'|}$. Combining the two and using the fact that $\cos\phi=\sin\psi$, we get: $\sinh |XY| \cosh |YY'| = \sin \psi \sinh |XY'|$. Using the hyperbolic law of sines a second time in the right triangle $XX'Y'$, we obtain $\frac{\sin \psi}{\sinh |X'Y'|}=\frac{1}{\sinh |XY'|}$. Substituting for $\sin \psi$ using this identity, we reach the desired equality. 
\end{proof}
We now get back to our proof of Lemma \ref{lemma:lengths} and start by establishing the following fact:

\begin{lemma}
\label{lemma:saccheri-both}
There exists $N\in\mathbb N$ such that, for all $n>N$, the quadrilateral $C'B'BC$ is a Saccheri quadrilateral with base $B'C'$ of length $2 |FE|$, symmetry line $DD'$ and base and summit midpoints $D'$ and $D$. 
\end{lemma}

\begin{proof}
In the spherical setting, we refer to Remark \ref{small} to select $N\in\mathbb N$ such that all the edge-lengths of $t_n$ and all heights are strictly less than $\frac{\pi}{2}$. In the hyperbolic setting, we select $N=0$. 

By construction, the triangles $FB'B$ and $FA'A$ share an angle and two edge-lengths, and are therefore congruent. Likewise for the triangles $EAA'$ and $ECC'$. This shows that $|AA'|=|CC'|=|BB'|$. Since $B'$ and $C'$ are the orthogonal projections of $B$ and $C$ on the line $FE$, there only remains to show that $C'B'BC$ is indeed a quadrilateral. This is clear in the hyperbolic case, but requires more care in the spherical case. 

In the spherical setting, we first observe that $B'C'$ and $BC$ cannot intersect. Indeed, $|B'C'|\leq\pi$, $|B'B|,|C'C|<\frac{\pi}{2}$ and  $\angle C'B'B=\angle B'C'C=\frac{\pi}{2}$. We cannot have $B'=B$ or $C=C'$ as it would imply that all points in fact lie on a single great circle. Therefore, it must be that $B$ and $C$ are two distinct interior points of the same right-angled spherical lune with one of its half great circles passing through $B'$ and $C'$ (see Fig. \ref{fig:lune}). However, the geodesic joining any two interior points of a right-angled spherical lune does not cross either boundary edge (a right-angled spherical lune can be completed to a hemisphere sharing either half great circle of the lune as its boundary great circle, and open hemispheres are convex). The only possibility is then that $BB'$ and $CC'$ intersect. But since $\angle C'B'B=\angle B'C'C=\frac{\pi}{2}$, this would imply that both $|BB'|$ and $|CC'|$ are greater than $\frac{\pi}{2}$, which contradicts our definition of orthogonal projections. This argument proves that $C'B'BC$ is a Saccheri quadrilateral. 

\begin{figure}[H]
	\begin{center}     
    	\includegraphics[width=0.4\textwidth]{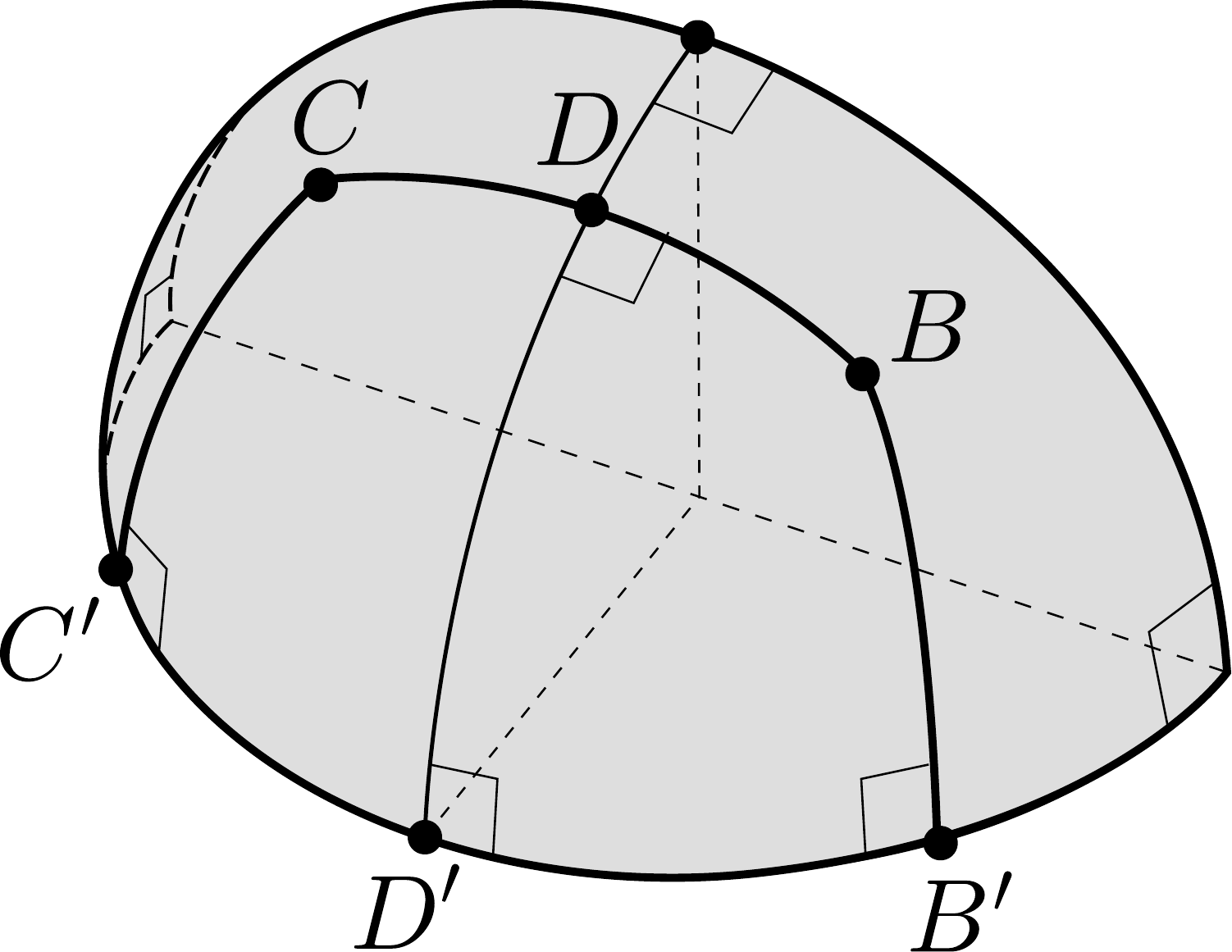}
        	\caption{}
        	\label{fig:lune}
	\end{center}
\end{figure}
\vspace{-1em}
We now prove that $D'$ is the midpoint of $B'C'$ and not its antipodal point. Denote by $D''$ the midpoint of $B'C'$. In the previous paragraph, we have shown all four points $B$, $C$, $B'$ and $C'$ to all lie in the same right-angled spherical lune. Since $C'B'BC$ is Saccheri, we now also know that the line through the midpoints of its base and summit is perpendicular to both base and summit (and is its only line of symmetry). Because of this, we can choose the right-angled lune containing all four points to have $D''$ as the midpoint of one of its two boundary half great circles. The edge $DD''$ can then be seen as a strict sub arc of the lune's equatorial arc (see Fig \ref{fig:lune}), which proves that $|DD''|<\frac{\pi}{2}$ and confirms that~$D''=D'$.

Lastly, we justify why the base $B'C'$ has length $2|FE|$. By construction, we have that $|B'C'|=|B'A'| + |A'C'|$ (resp. $|B'C'|=-|B'A'| + |A'C'|$, $|B'C'|=|B'A'| - |A'C'|$) when $\beta_n$ and $\gamma_n$ are acute (resp. when $\beta_n$ is obtuse, $\gamma_n$ is obtuse). Note that the case where both $\beta_n$ and $\gamma_n$ are obtuse is impossible since a spherical triangle with all its edge-lengths smaller than $\frac{\pi}{2}$ has at most one obtuse angle (this is a direct consequence of the spherical law of cosines).  Observe then that $|FE|=|FA'|+|A'E|$ (resp. $|FE|=-|FA'|+|A'E|$, $|FE|=|FA'|-|A'E|$) and we have $|B'A'|=2|FA'|$ and $|A'C'|=2|A'E|$ in all cases, since $2|FA'|$, $2|A'E|<\frac{\pi}{2}$. Thus, $|B'C'|=2|FA'|+2|A'E|=2|FE|$ (resp. $|B'C'|=-2|FA'|+2|A'E|=2|FE|$, $|B'C'|=2|FA'|-2|A'E|=2|FE|$).
\end{proof}

\begin{proof}[Proof of Lemma \ref{lemma:lengths}.] Using Lemma \ref{claim1}, we first show that the statement of Lemma \ref{lemma:lengths} is true for the sines of the edge lengths and the added constraint that $C\geq1$, namely:

\textbf{Claim.} \textit{ There exists an integer $N$ and a positive constant $1\leq C<2$ depending only on $t_0$ such that, for all $n>N$, $\sin a_{n+1} \leq C \> \sin \frac{a_n}{2}$.}\\
\textit{Proof of Claim.} The only cases of interest are the two non-trivial cases where $a_{n+1}$ is obtained as the parallel side of $a_n$ in $t_{n+1}$. In both of these cases, we can appeal to Remark \ref{small} and use Lemma \ref{lemma:saccheri-both} and the formulae of Lemma \ref{lemma:Lambert-Trigo} to obtain:

\[
\sin \frac{a_n}{2}=\sin |DC|=\sin |D'C'|\cdot \cos |CC'|=\sin a_{n+1}\cdot \cos |CC'|\tag{$\star$}
\]

If we now suppose by contradiction that our claim is false, then for all $N\in \mathbb N$ and for all $1<C<2$, there exists $n>N$ such that $\sin a_{n+1}>C\sin \frac{a_n}{2}$. Using $(\star)$, this implies that:

\[
\cos|CC'|=\frac{\sin\frac{a_n}{2}}{\sin a_{n+1}} < C^{-1}
\]

which in turn implies that:

\[|CC'| > \arccos C^{-1}>0\]

However, in positive curvature, the sides of a Lambert quadrilateral incident to the apex are strictly smaller than than their opposite side in the quadrilateral. Therefore, $|CC'| < |DD'|$ (using the Lambert quadrilateral $C'D'DC$). Applying the spherical version of Pythagoras' theorem in the right angled triangle $DD'F$, we see that $|DD'|\leq|FD|$, since $|DD'|,|DF|\leq\frac{\pi}{2}$. As Lemma \ref{claim1} guarantees that $|FD|\rightarrow 0$ when $n\rightarrow \infty$, we have shown that $|CC'|\rightarrow 0$ when $n \rightarrow \infty$. This contradicts $|CC'| > \arccos C^{-1}$ and proves the claim.

To get back to the proof of the lemma, we first note that, for all $x>0$, we have $\sin x<x$. On the other hand, for any $\epsilon>0$, we have $(1-\epsilon)x < \sin x$, for $x$ small enough. Using the particular value of $C$ and $N$ given by the previous claim, we can choose an $\epsilon>0$ small enough to guarantee that $C\slash (1-\epsilon)<2$. There is then an integer $N'>N$ large enough to guarantee that, for all $n>N'$:

\[
(1-\epsilon)a_{n+1}< \sin a_{n+1} < C \sin \left(\frac{a_n}{2}\right) < C \frac{a_n}{2}
\]

which gives the following desired inequality:

\[
a_{n+1} < \frac{C}{1-\epsilon} \frac{a_n}{2}
\]

and finishes the lemma, as $C\slash (1-\epsilon)<2$.

\end{proof}

We now get back to the proof of Theorem C, which we state again below.
\vspace{0.5em}

\noindent \textbf{Theorem C. }\textit{For any sequence of nested triangles $t_0,t_1.\ldots$ and for all $n\in\mathbb N$, there exists $l_a, L_a>0$ such that:
\[a_0\cdot l_a \leq  2^n \cdot a_n \leq a_0 \tag{\textit{Hyperbolic geometry}}\]
\[a_0 \leq 2^n \cdot a_n \leq a_0\cdot L_a \tag{\textit{Spherical geometry}}\]
and similarly for $b_n$, $c_n$. In addition, in the non-trivial cases where there exists at least some integer $n\in\mathbb N$ such that $a_{n+1}$ is obtained as the parallel side of $a_n$ in $t_{n+1}$, the inequalities are strict and $l_a$ (resp. $L_a$) approaches $1$ from below (resp. above) in the hyperbolic (resp. spherical) case as all the side lengths of $t_0$ become smaller.} 
\begin{proof}[Proof of Theorem C (Hyperbolic Setting)]
\textbf{Hyperbolic Upper Bound.} Note first that Lemma \ref{lemma:mid-edge} tells us that the sequence $(2^n a_n)_{n\in\mathbb N}$ is decreasing as $\frac{2^{n+1} a_{n+1}}{2^n a_n}=\frac{2 a_{n+1}}{a_n}\leq 1$. Since the sequence is bounded below by $0$, this guarantees its convergence to a non-negative limit. Applying Lemma \ref{lemma:mid-edge} $n$ times also directly gives us the $a_0$ upper bound:  $2^n a_n = 2^{n-1}(2 a_n)\leq 2^{n-1}a_{n-1} \leq \ldots \leq 2 a_1 \leq a_0$. 

\textbf{Hyperbolic Lower Bound.} Our task is to show that the sequence $(2^n a_n)_{n\in\mathbb N}$ converges to a strictly positive limit. Let us start by rewriting the previous limit as the following infinite product:

\[
\lim_{n\to\infty} 2^n \cdot a_n = \lim_{n\to\infty} \frac{2a_n}{a_{n-1}}\cdot \frac{2a_{n-1}}{a_{n-2}}\cdot\ldots \cdot \frac{2a_{1}}{a_0}\cdot a_0
\]

The convergence to a strictly positive value of this infinite product is equivalent to the finiteness of the infinite sum of the logarithm of its factors. Namely:

\[
a_0 \> \prod_{n=1}^{\infty} \frac{2a_{n}}{a_{n-1}}>0 \iff \left|\ln(a_0)\right| + \sum_{n=1}^{\infty}\left| \ln \left( \frac{2a_{n}}{a_{n-1}}\right)\right|=\left|\ln(a_0)\right|+\sum_{n=1}^{\infty} \ln\left(\frac{a_{n-1}}{2a_{n}}\right)<\infty
\]
Note that Lemma \ref{lemma:mid-edge} tells us the sign of the ratio inside the absolute value. In order to prove that this sum is indeed finite, we will show that the following inequality holds:
\[
\label{eqn:star}
\tag{$\star$}
\ln \left( \frac{a_{n-1}}{2a_{n}}\right)<\frac{b_n}{2}
\]

The convergence of the geometric series $(b_0\> 2^{-(n+1)})_{n\in\mathbb N}$ then ensures the convergence of the logarithm sum, as $\frac{b_n}{2}\leq b_0 \> 2^{-(n+1)}$. It then also provides us with a uniform bound:

\[
\sum_{n=1}^\infty \ln \left( \frac{2a_{n}}{a_{n-1}}\right)=-\sum_{n=1}^\infty \ln \left( \frac{a_{n-1}}{2a_{n}}\right)>-\sum_{n=1}^\infty \frac{b_n}{2}>-\frac{1}{2}\sum_{n=1}^\infty b_0\cdot 2^{-n} = - \frac{b_0}{2}
\]

This gives us that $l_a =e^{-\frac{b_0}{2}}$ is a valid choice, since after taking the exponential on both sides in the previous inequality, we obtain:

\[
2^n \cdot a_n > \lim_{n\to\infty} 2^n \cdot a_n = a_0\prod_{n=1}^{\infty} \frac{2a_{n}}{a_{n-1}} > a_0 \> e^{-\frac{b}{2}}
\]
Note that $l_a$ indeed approaches $1$ from below when all side lengths of $t_0$ become smaller (thus approaching the Euclidean case for small triangles). We now proceed to prove inequality (\ref{eqn:star}). Note that this inequality is trivially satisfied for the terms of the sequences where $a_{n+1}$ is contained in $a_n$, since the logarithms of their corresponding ratio are then each equal to zero and do not contribute to the sum. It is thus enough to prove it in the case where $a_{n+1}$ is obtained as the parallel side of $a_n$ in $t_{n+1}$.

We shall use the same notation as on Figure \ref{fig:construction-both}, with the added simplification that we will write $a$ in place of $a_{n}$ and $a'$ in place of $a_{n+1}$, and similarly for the other sides. We begin by extending the geodesic segment joining $E$ and $F$ on either side of the triangle and introduce points $G$ and $H$ such that $|GE|=|EF|=|FH|=a'$. The resulting figure, resembling a ``jester hat'', can be seen on Figure \ref{fig:bat}. The geodesic triangle $GEC$ and $FEA$ share two equal sides and an angle and are therefore congruent. For the same reason, the triangle $FHB$ is also congruent to $FEA$ and therefore also to $GEC$. 

\begin{figure}[H]
	\begin{center}     
    	\includegraphics[width=\textwidth]{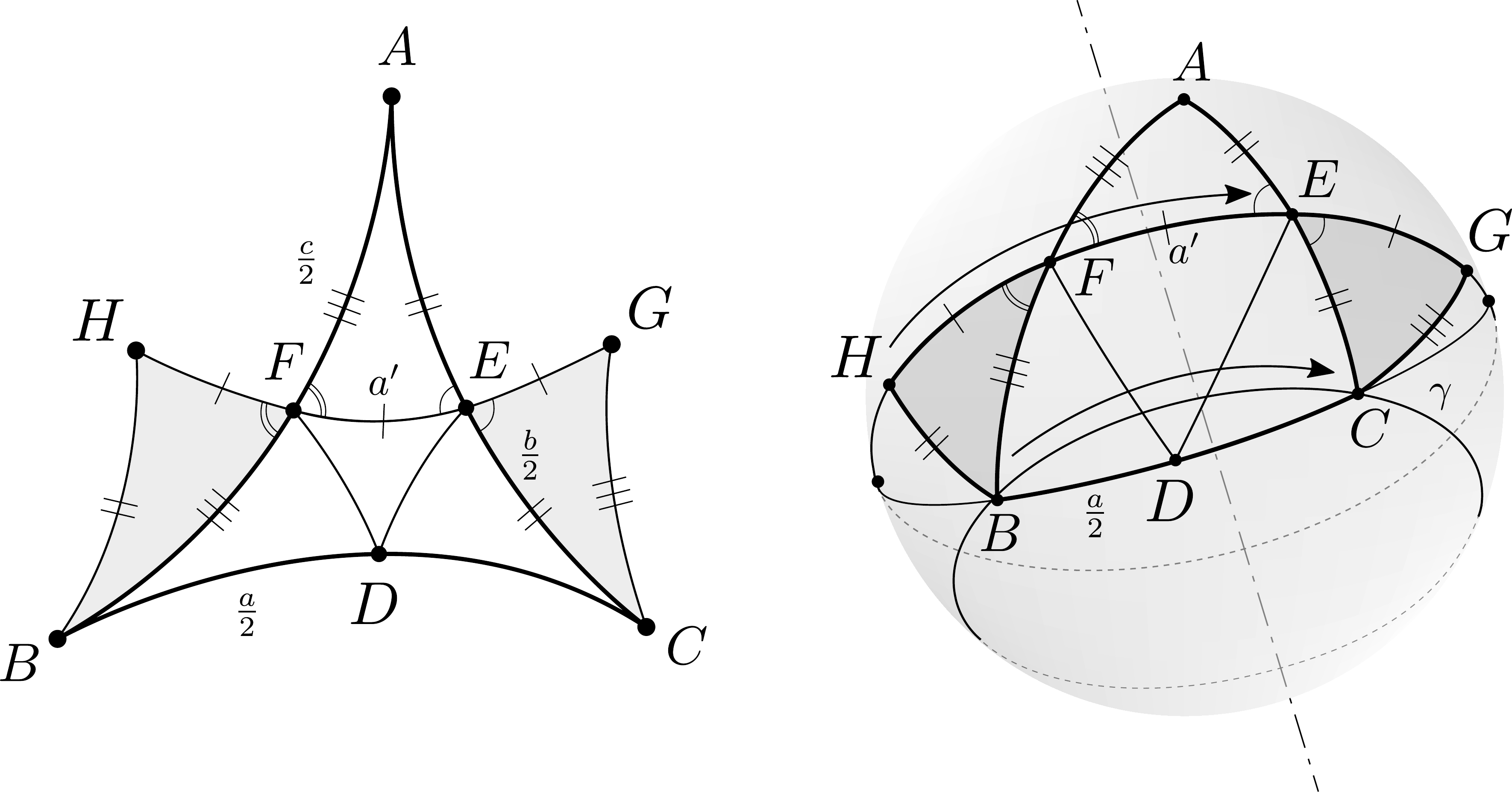}
        	\caption{: The hyperbolic and spherical ``jester hats''.}
        	\label{fig:bat}
	\end{center}
\end{figure}

\begin{rem}
\label{newproof}
We take this opportunity to give a short elegant proof of Lemma~\ref{lemma:mid-edge}, stating that $2a'<a$ in hyperbolic triangles and $2a'>a$ in spherical triangles. Indeed, with this construction we know that there is a hyperbolic translation, with axis the geodesic line going through $E$ and $F$ and translation distance $2a'$, taking the triangle $GEC$ to the triangle $FHB$. Since the minimum translation distance of a hyperbolic translation is realised for points on its axis, we can conclude that $a>2a'$. We can of course replicate the exact same construction in the spherical case, where the hyperbolic translation along the line $EF$ now becomes the rotation of the sphere whose axis has the great circle passing through $EF$ as its equator circle. Since the maximum translation distance of a rotation of the sphere is realised along its associated equator circle (see Fig. \ref{fig:bat}), this concludes the proof of Lemma~\ref{lemma:mid-edge}.
\end{rem}

We now turn our attention to the geodesic quadrilateral $EHBC$. There is a unique equidistant curve $\gamma$ staying within a fixed distance $h$ from the geodesic line $EF$ and passing through $B$ and $C$ (Figure \ref{fig:4}). Since $h<\frac{b}{2}$ and the length $L_{\gamma}(B,C)$ of $\gamma$ between $B$ and $C$ is more than that of the geodesic segment joining $B$ and $C$, it is enough to show that 

\[
\ln \left(\frac{L_{\gamma}(B,C)}{2a'}\right) < h
\]
\vspace{-0.8em}
\begin{figure}[H]
	\begin{center}     
    	\includegraphics[width=0.55\textwidth]{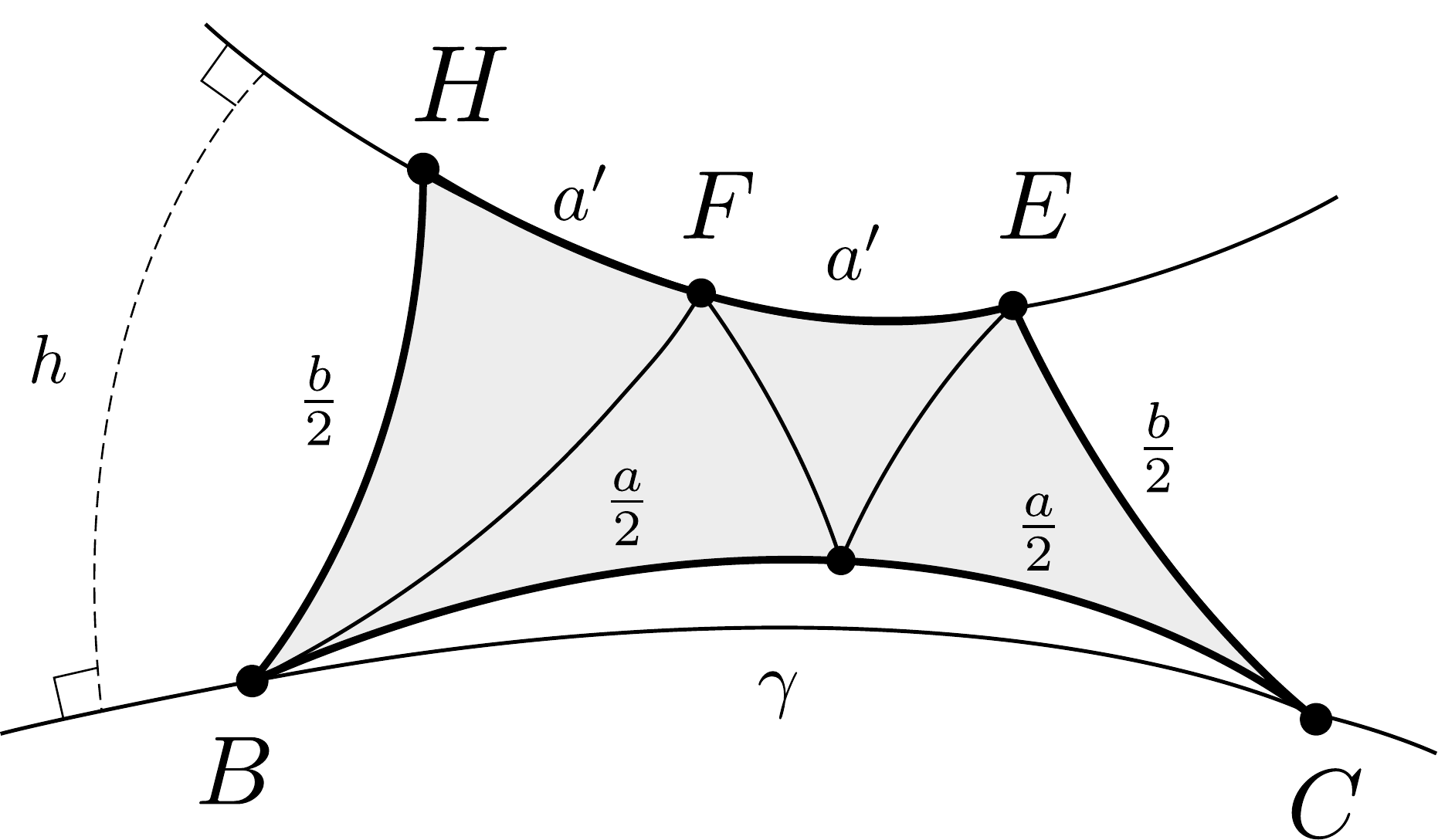}
        	\caption{}
        	\label{fig:4}
	\end{center}
\end{figure}

\noindent in order to show (\ref{eqn:star}). The key to the demonstration is now to notice that the figure in invariant under translation along $(EH)$, so that we may look at the ratio of infinitesimal displacements along $(EH)$ and $\gamma$ instead of that between $L_{\gamma}(B,C)$ and $2a'$. If we consider the quadrilateral $EHBC$ in the upper half-plane model (Figure \ref{fig:half-plane}), the equidistant curve $\gamma$ is now a line meeting the geodesic line $(EH)$ on the boundary. Denote by $\theta$ the angle between $(EH)$ and $\gamma$. Since the metric in the upper half-plane is scaled by the inverse of the $y$-coordinate, the ratio between infinitesimal displacements along $(EH)$ and $\gamma$ is simply given by $\cos ^{-1} \theta$. 

\begin{figure}[H]
	\begin{center}     
    	\includegraphics[width=0.65\textwidth]{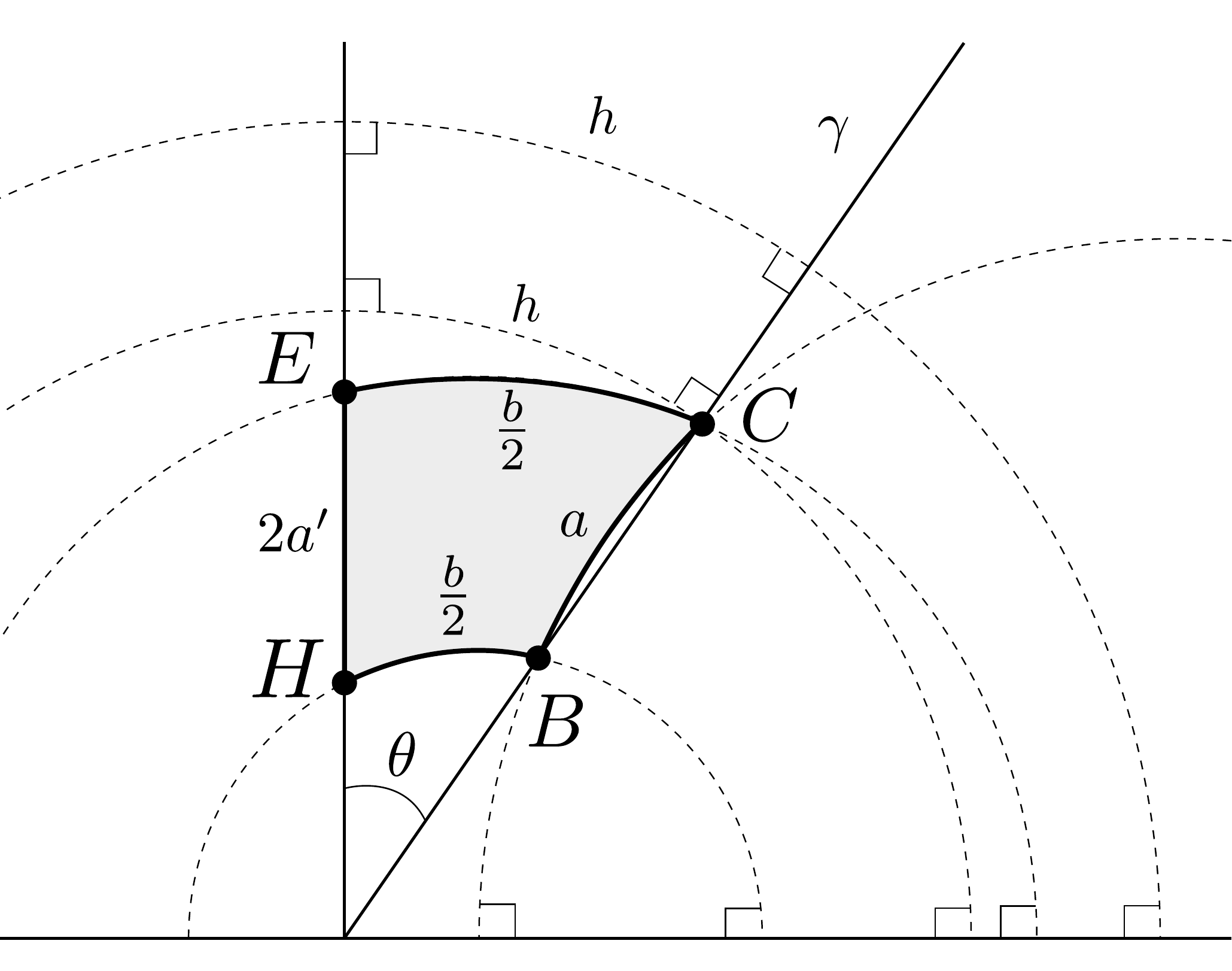}
        	\caption{}
        	\label{fig:half-plane}
	\end{center}
\end{figure}

On the other hand, we can compute the distance $h$ between the geodesic line $(EH)$ and the equidistant curve $\gamma$ using the metric of the upper half-plane model:

\[
h = \int_0^\theta \frac{1}{\cos t}dt
\]

Putting both together, it then remains to prove the following inequality:

\[
\ln \left(\frac{L_{\gamma}(B,C)}{2a'}\right) = \ln \left(\cos ^{-1}\theta \right)   < \int_0^\theta \frac{1}{\cos t}dt = h
\]

To obtain that inequality, it is enough to differentiate both sides with respect to $\theta$ and show that the derivatives verify the inequality. After differentiating and simplifying, we reach the following equivalent inequality:

\[
\sin \theta < 1
\]

Note that the previous inequality need in fact only be strict at one point to guarantee the strict inequality after integration (the point $\theta=0$, for example, is then enough). Alternatively, we can also observe that the only possibility for equality is when $\theta=\frac{\pi}{2}$. However, the line $EH$ and $\gamma$ are at finite distance $h$ from each other, which implies precisely that $\theta<\frac{\pi}{2}$.

\end{proof}

\begin{proof}[Proof of Theorem C (Spherical Setting)]
\textbf{Spherical Lower Bound.} We start with the opposite observation than in the hyperbolic case, namely that Lemma \ref{lemma:mid-edge} here
gives us the lower bound instead of the upper bound: $2^n a_n \geq a_0$. Likewise, it also informs us that the sequence $(2^n a_n)_{n\in\mathbb N}$ is increasing as $\frac{2^{n+1} a_{n+1}}{2^n a_n}=\frac{2 a_{n+1}}{a_n}\geq 1 $. Our task is thus to show that the sequence $(2^n a_n)_{n\in\mathbb N}$ is bounded above.

\textbf{Spherical Upper Bound.} In the spherical case, we give a shorter, purely trigonometric proof (note that this proof is also available in the hyperbolic case). It will prove easier to exhibit an upper bound for the sequence $(2^n  \sin a_n)_{n\in\mathbb N}$ instead, which will transfer to our original sequence via a small correction term. As in the hyperbolic setting, we start by rewriting the $n$-th term of the sequence as the following product:

\[
2^n \cdot \sin a_n = \frac{2\sin a_n}{\sin a_{n-1}}\cdot \frac{2\sin a_{n-1}}{\sin a_{n-2}}\cdot\ldots \cdot \frac{2\sin a_{1}}{\sin a_0}\cdot \sin a_0
\]

Finding an upper bound for this product of $n$ terms is again equivalent to finding an upper bound to the sum of the logarithm of its factors. Namely:

\[
a_0 \> \prod_{i=1}^{n} \frac{2\sin a_{i}}{\sin a_{i-1}}>\delta > 0 \iff \left|\ln(a_0)\right| + \sum_{i=1}^{n}\left| \ln \left( \frac{2 \sin a_{i}}{\sin a_{i-1}}\right)\right|<\Delta <\infty
\]
Where $e^\Delta = \delta$. There again, Lemma \ref{lemma:mid-edge} tells us the sign of the ratio inside the absolute value. To prove an upper bound on this sum of logarithms, we take up again the notation of Figure \ref{fig:construction-both} and focus on the non-trivial case where $a_{n+1}$ is obtained as the parallel side of $a_n$ in $t_{n+1}$. Lemma \ref{lemma:saccheri-both} guarantees that for a large enough $n$, the quadrilateral $C'B'BC$ is Saccheri and the quadrilateral $C'D'DC$ is Lambert. For such an $n$, we can now use the spherical Pythagoras theorem in the triangle $D'CC'$ and Lemma \ref{lemma:Lambert-Trigo} in $C'D'DC$ to obtain the two following identities:

\[ \cos |D'C| = \cos |CC'|\cdot \cos |D'C'| \tag{3.1}\]

\[ \sin |DC| = \sin |D'C'| \cdot \cos |CC'|  \tag{3.2}\]

Using equation (3.2) together with the double angle formula for the sine, we reach the following inequality:
\[ \frac{2 \sin \frac{a_{n+1}}{2}}{\sin \frac{a_n}{2}} =
\left(\cos \left( \frac{a_{n+1}}{2}\right) \cos |CC'|\right)^{-1}>1 \tag{3.3}\]

Combining equation (3.3) with equation (3.1), we obtain, for $t_n$ small enough:
\[
\frac{2 \sin \frac{a_{n+1}}{2}}{\sin \frac{a_n}{2}} =\frac{\cos a_{n+1}}{\cos |D'C|\cdot \cos \frac{a_{n+1}}{2}} < \left(\cos\frac{a_{n+1}}{2}\right)^{-1}\tag{3.4}
\]
The last inequality is obtained by noticing that, for small enough spherical triangles, the length of the hypotenuse $|D'C|$ in the right-angled spherical triangle $D'CC'$ is larger than that of the leg $|D'C'|=a_{n+1}$. This can be shown to be true of any spherical triangle contained in a spherical octant and thus, in particular, for any triangle with side lengths smaller than $\frac{\pi}{3}$. By Lemmas \ref{claim1} and \ref{lemma:lengths}, we can pick $N\in\mathbb N$ such that for all $n>N$, all three sides of $t_n$ are smaller than $\frac{\pi}{3}$ and the inequality $a_{n+1}\leq C a_n / 2$ holds. Since all side-lengths are strictly less than $\frac{\pi}{2}$, our choice of $N$ was large enough to guarantee that the quadrilateral $C'B'BC$ is Saccheri, allowing us to make use of the previous derivations. Writing $C_N=\sum_{i=1}^{N}\left(\frac{2 \sin a_{i+1}}{2}/\sin \frac{a_i}{2}\right)$, we can now write the following inequality:

\[ \sum_{i=1}^{n} \ln \left( \frac{2 \sin \frac{a_{i+1}}{2}}{\sin \frac{a_i}{2}}\right) <  - \sum_{i=N}^{n} \ln \cos \frac{a_{i+1}}{2} +C_N< \sum_{i=N}^{n} \frac{a_{i+1}}{2} +C_N\tag{3.5}\]
where the last inequality stems from the observation that $-\ln \cos (x) < x$ for $x\in (0,\frac{\pi}{3})$. Indeed, the first function is strictly convex and has vanishing derivative at $0$. It is then enough to check that the inequality is true in $\frac{\pi}{3}$: a quick computation gives $\cos \frac{\pi}{3}=\frac{1}{2}>\frac{1}{e}>\frac{1}{e^\frac{\pi}{3}}$, which is the desired inequality after taking logarithms and changing signs. 
Using Lemma \ref{lemma:lengths} and after multiplying inequality (3.5) by $\ln\sin\frac{a_0}{2}$ and taking the exponential, we  obtain the following chain of inequalities:
\[
2^n \sin \frac{a_n}{2} < \sin\frac{a_0}{2} \exp{\left(\frac{a_0}{2}\sum_{i=N}^{n} \left(\frac{C}{2}\right)^{i}+C_N\right)} < \sin\frac{a_0}{2}  \exp{\left(a_0\frac{ \left(C/2\right)^{N}}{2-C}+C_N\right)}
\]

For all $n\in\mathbb N$, let us consider the quantity $\epsilon_n = a_n/\sin a_n -1$. it is easy to see that this quantity is always strictly positive, goes to $0$ as $n\rightarrow \infty$ and $(\epsilon_n)_{n\in\mathbb N}$ is a monotone decreasing sequence. Noting that $\sin a_n < 2\sin \frac{a_n}{2}$ and $\sin \frac{a_0}{2}  < \frac{a_0}{2}$, we obtain:

\[
2^n a_n = 2^n (1+\epsilon_n)  \sin a_n < (1+\epsilon_0) \> a_0\> \exp{\left(a_0\frac{ \left(C/2\right)^{N}}{2-C}+C_N\right)}
\]

From this, we conclude that $l_a =(1+\epsilon_0)\exp{\left(a_0 \left(C/2\right)^{N}/(2-C)+C_N\right)}$ is a valid choice. Indeed, by construction, both $N$ and $C_N$ become zero for $a_0$ small enough. Since $C$ is fixed, we also have that both $(1+\epsilon_0)$ and the exponential term approach~$1$ from above when $a_0$ becomes small.
\end{proof}

\section{Stabilisation of Heights}
\label{sec4}

In this section, we investigate the behaviour of heights in the subdivision and show that they also ``stabilise" to the Euclidean case as our triangulations refine. The exact meaning of this expression is made precise by the statement of the following proposition:

\begin{prop}
\label{lemma:altitude}
For any sequence of nested triangles $t_0,t_1,\ldots$, and for all $n\in\mathbb N$, there exists $l_h, L_h>0$ such that:
\[ h_0\cdot l_h \leq 2^n \cdot h_n \leq h_0 \cdot L_h\]
where $h_n$ denotes the height from the vertex incident to $\alpha_n$ onto the line prolonging the side of length $a_{n}$ in $t_n$. In addition, in the non-trivial cases where there exists at least some integer $n\in\mathbb N$ such that $a_{n+1}$ is obtained as the parallel side of $a_n$ in $t_{n+1}$, the inequalities are strict and $l_h$ (resp. $L_h$) approaches 1 from below (resp. above) as all the side lengths of $t_0$ become smaller. 
\end{prop}

The following lemma will prove to be useful to simplify our proof of Proposition~\ref{lemma:altitude}:

\begin{lemma}
\label{lemma:trick}
In the hyperbolic setting (resp. in the spherical setting, for $n$ large enough), the height of the vertex incident to $\alpha_{n+1}$ to the line prolonging the side of length $a_{n+1}$ in $t_{n+1}$ is minimal among all 4 choices of $t_{n+1}$ (resp. maximal) when the innermost triangle of $t_n$ is selected.
\end{lemma}

\begin{proof}
We stated the theorem and give the proof only for the height associated to $\alpha_{n+1}$ in the hyperbolic setting, the spherical proof is obtained by simply reversing each conclusion/inequality and the other cases are derived in the same fashion. Recall that $|A'A|=|C'C|$ and consider the Lambert quadrilateral $D'C'CD$. In the hyperbolic case, we have that $|C'C|>|DD'|$, as the sides incident to the apex are larger than their opposite sides in hyperbolic Lambert quadrilaterals (this opposite conclusion is true for spherical Lambert quadrilaterals). The same inequality can be derived for each of the other two heights by using the Lambert quadrilaterals $FDD'F$ and $DE'ED'$.
\end{proof}

\begin{proof}[Proof of Proposition \ref{lemma:altitude} (Hyperbolic Setting)]

There again, we shall only give the proof regarding the heights associated to $\alpha_n$, the proofs of the other cases are derived in the same fashion. We once again take up our notation for Lemma \ref{lemma:lengths}. In addition, let $A''$ be the orthogonal projection of $A$ onto the line $BC$ and denote by $E''$ the midpoint of $A''C$ (see Fig. \ref{fig:bat-lines}). For the sake of brevity, we shall disregard the similar but easier case where $A''=C$. 

\begin{figure}[H]
	\begin{center}     
    	\includegraphics[width=0.65\textwidth]{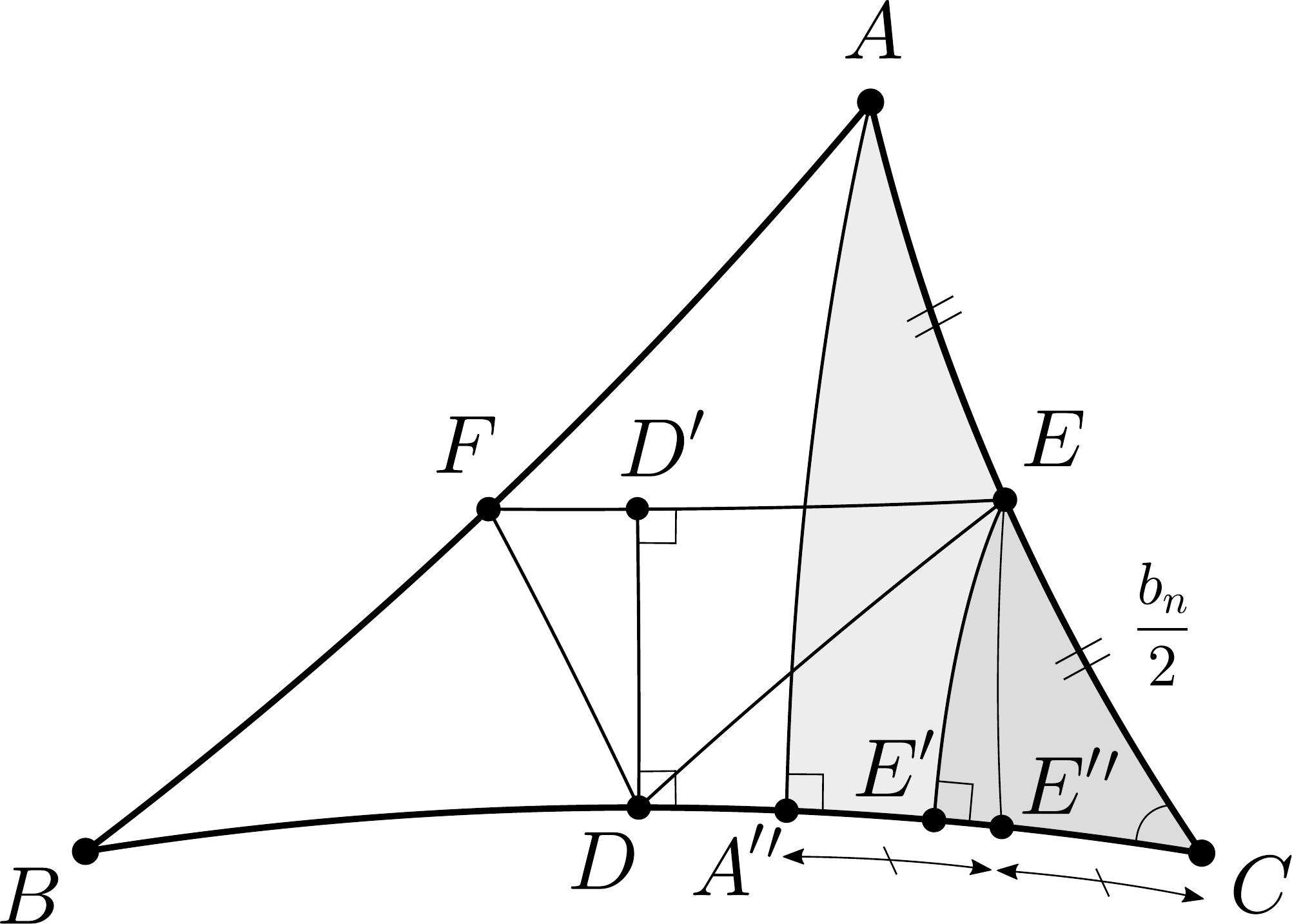}
        	\caption{}
        	\label{fig:bat-lines}
	\end{center}
\end{figure}

Once again, we begin by writing the quantity $2^n h_n$ as the following product:

\[
2^n \cdot  h_{n} = \frac{2 h_{n}}{h_{n-1}}\cdot \frac{2 h_{n-1}}{h_{n-2}}\cdot\ldots \cdot \frac{2 h_{1}}{h_{0}}\cdot h_{0} = h_{0} \> \prod_{n=1}^{n} \frac{2 h_{i}}{h_{i-1}}
\]

\textbf{Hyperbolic Upper Bound.} Note first that, $\frac{1}{2}|A''A|>|EE''|$ by Lemma \ref{lemma:mid-edge}. Using the hyperbolic Pythagoras theorem we then have that $|EE''|>|EE'|$. Lemma \ref{lemma:trick} then guarantees that $|EE'|>|DD'|$. Putting everything together, we obtain that $\frac{1}{2}h_n = \frac{1}{2}|A''A|>|EE''|>|EE'|>|DD'|$, which shows that each factor (other than $h_0)$ in the previous product is strictly less than 1. This shows the upper bound in the two cases where $t_{n+1}=CED$ (where $h_{n+1}=|EE'|$) and $t_{n+1}=FDE$ (where $h_{n+1}=|DD'|$). The case where $t_{n+1}=DFB$ is symmetrical to the case $t_{n+1}=CED$, but a slightly more circumvoluted argument is required to derive the upper bound when $t_{n+1}=EAF$. Let us denote by $\gamma$ the angle $\angle E'CA$ and by $\gamma'$ the angle $\angle A'EA=\angle C'EC$ (see Fig. \ref{fig:heights-case-3}). Note that $\gamma=\gamma_n$ and $\gamma'=\angle FEA$ if $\gamma_n \leq \frac{\pi}{2}$, but otherwise we instead have $\gamma=\pi -\gamma_n$ and $\gamma'=\pi - \angle FEA$. Applying hyperbolic trigonometric identities in the triangles $EE'C$ and $ECC'$, we obtain:

\[
\sinh |EE'| = \sin \gamma \cdot \sinh |EC|\tag{4.1}
\]
\[
\sinh |CC'| = \sin \gamma' \cdot \sinh |EC|\tag{4.2}
\]
\begin{figure}[H]
	\begin{center}     
    	\includegraphics[width=0.55\textwidth]{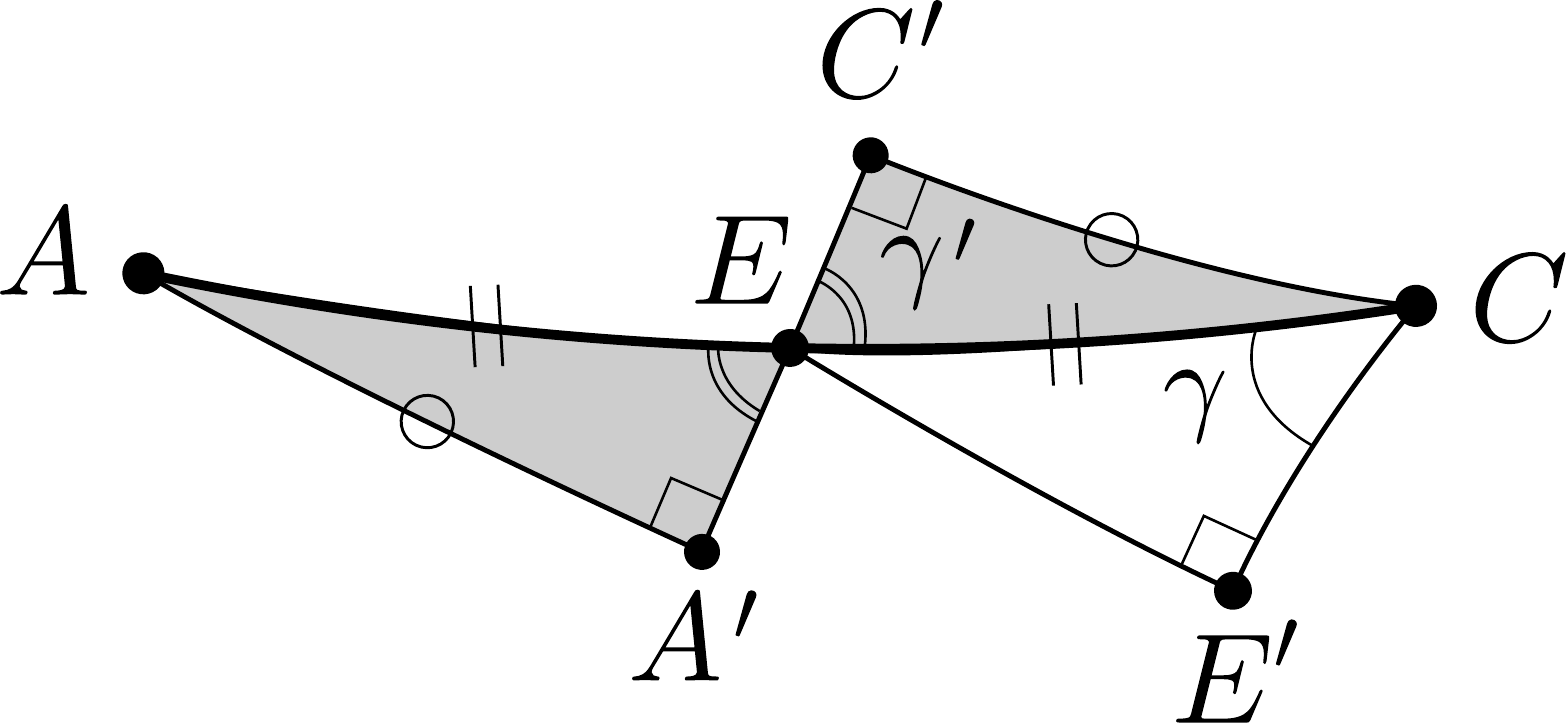}
        	\caption{}
        	\label{fig:heights-case-3}
	\end{center}
\end{figure}
Using the hyperbolic sine law in the triangles $ABC$ and $AFE$ and making use on their shared angle $\alpha_n$, we get:

\[
\sin \gamma' = \frac{\sinh \frac{c_n}{2}}{\sinh c_n} \frac{\sinh a_n}{\sinh a_{n+1}} \sin \gamma \tag{4.3}
\]

Let us denote by $\rho_{n+1}$ the factor preceding $\sin \gamma$ in the previous equation, i.e. such that we have $\sin \gamma' = \rho_{n+1} \sin \gamma$. Combining equations (4.1), (4.2) and (4.3) now gives us:
\[
\sinh |CC'|=\rho_{n+1} \> \sinh|EE'| \tag{4.4}
\]
Using the observation that for all $x>0$, we have $\sinh \frac{x}{2}<\frac{1}{2}\sinh x$,  we reach:

\[
\rho_{n+1} < \frac{\sinh \frac{a_n}{2}}{2 \sinh \frac{a_{n+1}}{2}}\> \cosh \frac{a_n}{2}\tag{4.5}
\]
Using inequality (\ref{eqn:star}) in the proof of Theorem C and noting that, for all $x>0$, we have $\sinh x > x$ and $\sinh x / x < \exp x$, we obtain that:

\[
\frac{\sinh \frac{a_n}{2}}{2 \sinh \frac{a_{n+1}}{2}} < \frac{\sinh \frac{a_n}{2}}{a_{n+1}} = \frac{a_n}{2\: a_{n+1}}\frac{\sinh \frac{a_{n}}{2}}{\frac{a_{n}}{2}}<  \exp \left(\frac{b_{n+1}}{2}\right) \frac{\sinh \frac{a_{n}}{2}}{\frac{a_{n}}{2}}<\exp \left(\frac{b_{n+1}}{2}+\frac{a_n}{2}\right)
\]
which leads us to the following upper bound on $\rho_n$ (noting that for all $x>0$, we have $\cosh x < \exp x$):

\[\rho_{n+1} < \exp \left(\frac{b_{n+1}}{2}+\frac{a_n}{2}\right)\cosh \frac{a_n}{2} < \exp \left(\frac{b_{n+1}}{2}+a_n\right)
\]
Observing that, for all $n>0$, we have $\exp \left(\frac{b_{n+1}}{2}+a_n\right)>1$, we obtain the following inequality (for any choice of $t_{n+1}$):
\[
2 \sinh h_{n+1} < 2 \max\{\rho_{n+1},1\} \sinh |EE'| < \exp \left(\frac{b_{n+1}}{2}+a_n\right) \sinh h_n
\]
From this we conclude that:

\[
2^n h_n < 2^n \sinh h_n = \sinh h_0 \prod_{i=0}^n \frac{2 \sinh h_{i+1}}{\sinh h_i}<\sinh h_0 \prod_{i=0}^n \exp \left(\frac{b_{i+1}}{2}+a_i\right) 
\]
Making use of Lemma \ref{lemma:mid-edge}, we obtain:

\[
2^n h_n < \sinh h_0 \cdot \exp \left(\sum_{i=0}^n \left(\frac{b_{i+1}}{2} + a_i\right)\right) < \frac{\sinh h_0}{h_0} \> e^{\frac{b_0}{2}+2 a_0} h_0
\]
We have thus shown that $L_h=\frac{\sinh h_0}{h_0} \> e^{\frac{b_0}{2}+2 a_0} > 1$ is a valid choice as $L_h$ indeed approaches $1$ from above as the edge lengths of $t_0$ become smaller. 

\textbf{Hyperbolic Lower Bound.} We now want a lower bound for the product $\prod_{n=1}^{n} \frac{2 h_{i}}{h_{i-1}}$. Because of Lemma \ref{lemma:trick}, we know that it is enough to derive the lower bound in the case where $t_{n+1}$ is obtained as the innermost triangle. Once again, it is equivalent to derive an upper bound  for the sum of the absolute values of the logarithm of its factors, namely: 

\[
2^n \cdot  h_{n} = h_{0} \> \prod_{i=1}^{n} \frac{2 h_{i}}{h_{i-1}}>\delta >0 \iff \left|\ln(h_{0})\right| + \sum_{i=1}^{n}\left| \ln \left( \frac{2 h_{i}}{h_{i-1}}\right)\right|<\Delta < \infty
\]
Where $e^{\Delta}=\delta$. However, it will prove easier to bound the ratio $\sinh h_n / \sinh h_{n-1}$ instead. We start by applying Lemma \ref{lemma:Lambert-Trigo} in the Lambert quadrilateral $DE'ED'$:

\[
\sinh h_{n+1} = \sinh |DD'| = \frac{\sinh |E'E|}{\cosh |D'E|} \label{eq:4.6}\tag{4.6}
\]

We can also establish the following identity through the hyperbolic sine rule applied to the triangles $CAA''$ and $CEE'$:

\[
 \frac{\sinh h_{n}}{\sinh |EE'|} = \frac{\sinh |A''A|}{\sinh |EE'|} = \frac{\sinh |CA|}{\sinh |CE|} = \frac{\sinh 2b_n}{\sinh b_n} 
\label{eq:4.7}\tag{4.7}\]
Combining both identities (\ref{eq:4.6}) and (\ref{eq:4.7}), we obtain:

\[
\frac{2 \sinh h_{n+1}}{\sinh h_{n}} = \frac{2 \sinh b_n\cdot \sinh |EE'|}{\sinh 2b_n \cdot \sinh |EE'|\cdot \cosh |ED'|} = \frac{1}{\cosh b_n\cosh|ED'|}
\tag{4.8}
\]
This shows that $2 \sinh h_{n+1} / \sinh h_{n} <1$ and  tells us that we need a lower bound on the ratios $2 \sinh h_{i}/ \sinh h_{i-1}$ instead of an upper bound as $\left| \ln \left( 2 \sinh h_{i} / \sinh h_{i-1}\right)\right|=\ln \left( \sinh h_{i-1} / 2 \sinh h_{i}\right)$. Since $|ED'| \leq b_{n+1}$, as the leg is always less than the hypotenuse in hyperbolic geometry, we obtain:

\[
\frac{2 \sinh h_{n+1}}{\sinh h_{n}} \geq \frac{1}{\cosh b_n \cdot \cosh b_{n+1}} \geq (\cosh b_n)^{-2}
\label{eq:4.8}\tag{4.9}
\]

As $\ln \cosh b_n  < b_n$ and the sequence $(b_n)_{n\in \mathbb N}$ is bounded above by the geometric series $(b_0\cdot 2^{-n})_{n\in \mathbb N}$, inequality (\ref{eq:4.8}) gives us the desired logarithm convergence criterion for $\sinh h_n$:

\[
\sum_{i=0}^{n}\left| \ln \left( \frac{2 \sinh h_{i+1}}{\sinh h_{i}}\right)\right| <2 \sum_{i=0}^n \ln (\cosh b_{i}) <2 \sum_{i=0}^n b_{i} 
\]
Taking the exponential of both sides and multiplying by $\sinh h_0$, we now get back to the original product:

\[
2^n \sinh h_n = \sinh h_0\>\prod_{i=0}^{n}\frac{2 \sinh h_{i+1}}{\sinh h_{i}} >  \sinh h_0 \> \exp\left(-\> 2b_0\>  \sum_{i=1}^{n} 2^{-i}\right) > \sinh h_0 \> e^{-2 b_0}
\]

For any $n\in\mathbb N$, we can choose $\epsilon_n = (h_n-\sinh h_n)/\sinh h_n$ so that $h_n>(1-\epsilon_n) \sinh h_n $. It is easy to see that, for all $n\in\mathbb N$, $\epsilon_n >0$ and $\lim_{n\to \infty}\epsilon_n \rightarrow 0$. Moreover, one can show the sequence $(\epsilon_n)_{n\in\mathbb N}$ to be monotone decreasing. This allows us to write, for all $n\in\mathbb N$:

\[2^n h_n > 2^n \sinh h_n (1-\epsilon_{n}) >\lim_{n\to\infty}2^n \>  \sinh h_{n}(1-\epsilon_{0}) > \sinh h_0 \> e^{-2 b_0} (1-\epsilon_{0})\]

As $t_0$ becomes smaller, both $\exp(-2b_0)$ and $(1-\epsilon_0)$ approach 1 from below. Taking $l_h=\exp(-2 b_0) (1-\epsilon_{0})$ and noticing that $\sinh h_0 > h_0$ thus finishes the proof.  
\end{proof}

\begin{proof}[Proof of Proposition \ref{lemma:altitude} (Spherical Setting)]
Some of the arguments used to derive the hyperbolic upper bound unfortunately do not translate to the spherical lower bound, which is why we take a slightly different approach here. 

\textbf{Spherical Lower Bound.} In the proof of the hyperbolic lower bound, we showed equality (4.8) using only hyperbolic trigonometry, this result is thus also valid in the spherical case (switching hyperbolic functions for the spherical ones) as long as Lemma \ref{lemma:saccheri-both} holds. Selecting the integer $N$ given by Lemma \ref{lemma:saccheri-both}, we have that, for all $n>N$ for which $t_{n+1}$ is obtained as the innermost triangle of $t_n$:

\[
\frac{2\sin h_{n+1}}{\sin h_n}=\frac{1}{\cos b_n\cos|ED'|}\tag{4.10}
\]

 For such values of $n$, we thus have $2 \sin h_{n+1} / \sin h_n>1$. 

In the case where $t_{n+1}=CED$, we make use of the spherical equivalent of identity (4.6) and the spherical Pythagoras theorem in triangle $DED'$ to obtain, for all $n>N$:

\[
\sin h_{n+1} = \sin |DD'| \> \cos |D'E| > \frac{\sin h_n}{2} \>\frac{\cos |DE|}{\cos |DD'| } > \frac{\sin h_n}{2} \> \cos c_{n+1} \tag{4.11}
\]
which shows that for such values of $n$, we have $2 \sin h_{n+1} / \sin h_n>\cos c_{n+1}$. The case where $t_{n+1}=DFB$ is symmetrical. The case where $t_{n+1}=EAF$ is dealt with in the exact same way as the hyperbolic upper bound case, substituting spherical functions for the hyperbolic ones, reversing all the inequalities used and making use of (3.4) instead of (\ref{eqn:star}). In this fashion, we reach the inequality:

\[
\sin h_{n+1} > \sin |EE'|\>\cos \frac{a_n}{2}\>\cos \frac{a_{n+1}}{2}>\frac{\sin h_n}{2}\cos c_{n+1}\left(\cos \frac{a_n}{2}\right)^2
\]
Putting all the cases together and observing that, for all $n>N$, we have $\cos c_{n+1}\left(\cos \frac{a_n}{2}\right)^2<\cos c_{n+1}<1$, we obtain the following inequality (for all choices of $t_{n+1})$:

\[
{\textstyle 2 \sinh h_{n+1} > \min \bigg\{{\cos c_{n+1}\left(\cos \frac{a_n}{2}\right)^2,\cos c_{n+1},1 }\bigg\}\sin h_n = \cos c_{n+1}\left(\cos \frac{a_n}{2}\right)^2 \sin h_n }
\]
Noting that for all $x>0$, we have $x>\sin x$, we obtain:

\[2^n h_{n} > 2^n \sin h_{n}  = \sin h_{0} \> \prod_{i=0}^{N-1} \frac{2 \sin h_{i+1}}{\sin h_{i}}\prod_{i=N}^{n} \frac{2 \sin h_{i+1}}{\sin h_{i}} > C_N \> \sin h_0 \prod_{i=N}^{n}\frac{2 \sin h_{i+1}}{\sin h_{i}}\]
where $C_N = \prod_{i=0}^{N-1} \frac{2 \sin h_{i+1}}{\sin h_{i}}$. Combining this last inequality with our lower bound for $2 \sin h_{i+1}/\sin h_{i}$, we obtain:

\[2^n h_{n} >  C_N \> \sin h_0 \prod_{i=N}^{n}\cos c_{i+1}\left(\cos \frac{a_i}{2}\right)^2\]
Adjusting our choice of $N$ to ensure that, for all $n>N$, we have $\cos c_{n+1},\> \cos \frac{a_n}{2}<\frac{\pi}{3}$, we can make use of the inequality $\cos x > e^{-x}$, for all $x\in(0,\frac{\pi}{3})$. Using Theorem C, we obtain:

\[2^n h_{n} >  C_N \> \sin h_0 \exp \>\left( - \sum_{i=0}^n (c_{i+1}+a_i) \right)>\frac{\sin h_0}{h_0} C_N e^{-(c_0 L_c + 2 a_0 L_a)} h_0\]
We have thus shown that $l_h=\frac{\sin h_0}{h_0} C_N e^{-(c_0 L_c + 2 a_0 L_a)}<1$ is a valid choice as $l_h$ indeed approaches 1 from below as the edge lengths of $t_0$ become smaller.

\textbf{Spherical Upper Bound.} Just like for the lower bound in the hyperbolic setting, Lemma \ref{lemma:trick} allows us to deal only with the case where $t_{n+1}$ is the innermost triangle of $t_n$. Since we showed that for all $n>N$, we have $2 \sin h_{n+1} / \sin h_n>1$, we now need an upper bound as $\left| \ln \left( 2 \sin h_{i+1} / \sin h_{i}\right)\right|=\ln \left( 2 \sin h_{i} / \sin  h_{i-1}\right)$. But since the cosine function is decreasing on $[0,\pi]$, we still desire an upper bound on $|ED'|$ in (4.10). Our previous remark that $|ED'| \leq b_{n+1}$, for all $n>N$, is again enough to yield the desired upper bound:
\[
\frac{2 \sin h_{n+1}}{\sin h_{n}} \leq \frac{1}{\cos b_n \cdot \cos b_{n+1}} \leq (\cos b_{n+1})^{-2}
\label{eq:4}\tag{4.12}
\]

Lemma \ref{lemma:lengths} allows us to adjust our choice of $N$ such that the series $(b_n)_{n>N}$ is bounded above by the geometric series $(b_0\cdot (\frac{C}{2})^{n})_{n>N}$. Adjusting one last time our choice of $N$ to ensure that all edge lengths are smaller than $\pi / 3$, we can once again guarantee that $-\ln \cos b_n  < b_n$, for all $n>N$. Writing $C'_N=\sum_{n=0}^{N}\ln\left(2 \sin h_{n+1}/\sin h_n\right)$, we can write:

\[
\sum_{i=0}^{n}\left| \ln \left( \frac{2 \sin h_{i+1}}{\sin h_{i}}\right)\right| <-2 \sum_{i=N}^n \ln (\cos b_{i+1})+C'_N < 2 \sum_{i=N}^n b_{i+1} +C'_N 
\]

which gives us:

\[
\left|\ln(h_0) \right|+ \sum_{i=0}^{n}\left| \ln \left( \frac{2 \sin h_{i+1}}{\sin h_{i}}\right)\right| < \left|\ln(h_0) \right| + 2b_0 \sum_{i=N}^{n} \left(\frac{C}{2}\right)^{i}+C'_N
\]
After taking the exponential of both sides, we retrieve:
\[
2^n \sin h_n = h_0 \prod_{i=0}^n \frac{2\sin h_{i+1}}{\sin h_i} < h_0\exp\left(4b_0\frac{(C/2)^N}{2-C}+C'_N\right)
\]

Let then $\epsilon_n = h_n / \sin h_n -1$. It is straightforward to check that for all $n$, we have $\epsilon_n>0$ and $\lim_{n\to\infty}\epsilon_n\rightarrow 0$. It is also easy to establish that $(\epsilon_n)_{n\in\mathbb N}$ is a monotone and decreasing sequence. Because of these observations, we can write:

\[
2^n h_n < (1+\epsilon_n)\> 2^n \sin h_n < h_0 \> (1+\epsilon_0) \exp\left(4b_0\frac{(C/2)^N}{2-C}+C'_N\right)
\]

Choosing $l_h=(1+\epsilon_0) \exp\left(4 b_0\frac{(C/2)^N}{2-C}+C'_N\right)$ finishes the proof. Indeed, both $N$ and $C'_N$ converges to zero for $t_0$ small enough and the exponential goes to 1 as $b_0$ grows smaller and $C$ is fixed. Both $(1+\epsilon_0)$ and the exponential approach $1$ (the exponent is always strictly positive).
\end{proof}

\section{Stabilisation of Angles}
\label{sec5}

In this section, we quickly derive the proof of our main result as a corollary of Proposition \ref{lemma:altitude} and Theorem \ref{thm:lengths} and show that the angles in the subdivision behave ``nicely'' in the limit, in a sense made precise by the following proposition:

\vspace{1em}

\noindent\textbf{Theorem B. }\textit{For any sequence of nested triangles $t_0,t_1, \ldots$ and for all $n\in\mathbb N$, there exists $l_\alpha,L_\alpha>0$ such that:
\[ \alpha_0 \cdot l_{\alpha} < \alpha_n < \alpha_0\cdot L_\alpha \]
In addition, $l_\alpha$ (resp. $L_\alpha$) approaches $1$ from below (resp. above) as all the side lengths of $t_0$ become smaller.}

\begin{proof}
We derive the proof in the hyperbolic case, the spherical case is obtained in the exact same fashion by swapping hyperbolic trigonometric functions for spherical ones and reversing the appropriate inequalities. 

Let $h_n$ denote the height of the altitude drawn from the vertex incident to $\alpha_n$ onto the line prolonging the side of $t_n$ of length $c_n$. Trigonometric identities (\cite{WT},~p.81) give us that $\sin \alpha_n$ is simply the ratio of $\sinh h_n$ by $\sinh b_n$. Let $\epsilon_n=1-\sinh (h_0 \>l_h\> 2^{-n})/(l_h \> 2^{-n}\sinh h_0)$. It is easy to check that $\epsilon_n >0$ for all $n\in \mathbb N$, and $\lim_{n\to\infty} \epsilon_n = 0$. Similarly, we define $\epsilon'_n$ as the quantity $1-\sinh (b_0 \>l_b\> 2^{-n})/(l_b \> 2^{-n}\sinh b_0)$ and make the same observations. We remark as well, that, for all $x>0$ and all $0<\epsilon <1$, we have $\sinh (\epsilon x) < \epsilon \sinh x$. Putting the previous observations together with Theorem \ref{thm:lengths} and Proposition \ref{lemma:altitude} gives us the two following two chains of inequalities:

\[
\sin \alpha_n 
=
 \frac{\sinh{h_n}}{\sinh{b_n}}
 < 
 \frac{\sinh(h_0 \> 2^{-n})}{\sinh ( b_0 \> l_b \> 2^{-n})}
 <
 \frac{(\sinh h_0)\>2^{-n}}{(1-\epsilon'_n)\>l_b \> 2^{-n}\sinh b_0}
 =
\sin \alpha_0 (1-\epsilon'_n)^{-1}l_b^{-1}
\]

\[
\sin \alpha_n 
=
 \frac{\sinh{h_n}}{\sinh{b_n}}
 >
  \frac{\sinh ( h_0 \> l_h \> 2^{-n})}{\sinh(b_0 \> 2^{-n})}
 >
  \frac{(1-\epsilon_n)\>l_h \> 2^{-n}\sinh h_0}{(\sinh b_0)\>2^{-n}}
 =
\sin \alpha_0 (1-\epsilon_n)l_h
\]

It is easy to show that both $(\epsilon_n)_{n\in\mathbb N}$ and $(\epsilon'_n)_{n\in\mathbb N}$ are monotone sequences decreasing to 0. Putting together both previous inequalities, we obtain:

\[
\sin \alpha_0 (1-\epsilon_0)l_h
<
\sin \alpha_n 
<
\sin \alpha_0\left((1-\epsilon'_0)l_b\right)^{-1}\tag{5.1}
\]

Choosing $l'_\alpha = (1-\epsilon_0)l_h$ and $L'_\alpha =\left((1-\epsilon'_0)l_b\right)^{-1}$ then yields the desired inequality on the sines of the angles. Indeed, both $1-\epsilon_0$ and $l_h$ approach $1$ from below as $t_0$ becomes smaller. Likewise, both $(1-\epsilon'_0)$ and $l_b$ approach $1$ from above as $t_0$ becomes smaller. 

While the sine function is monotonous and continuous, each element of the open interval $(0,1)$ has two pre-images under it, one in the interval $(0,\frac{\pi}{2})$ and the other in the interval $(\frac{\pi}{2},\pi)$. The following claim will allow us to differentiate between these two pre-images and establish the equivalent of inequality (5.1) for the angles: 
\vspace{0.5em}

\textbf{Claim. }\textit{In the hyperbolic setting, for any sequence of nested triangles $t_0,t_1, \ldots$ such that $\alpha_0>\frac{\pi}{2}$, and for all $n\in\mathbb N$, we have $\alpha_n > \frac{\pi}{2}$. In the spherical setting, for any sequence of nested triangles $t_0,t_1, \ldots$, there exists $N\in\mathbb N$  such that if $\alpha_N<\frac{\pi}{2}$, for all $n>N$, we have $\alpha_n < \frac{\pi}{2}$.}

\begin{figure}[H]
	\begin{center}     
    	\includegraphics[width=0.6\textwidth]{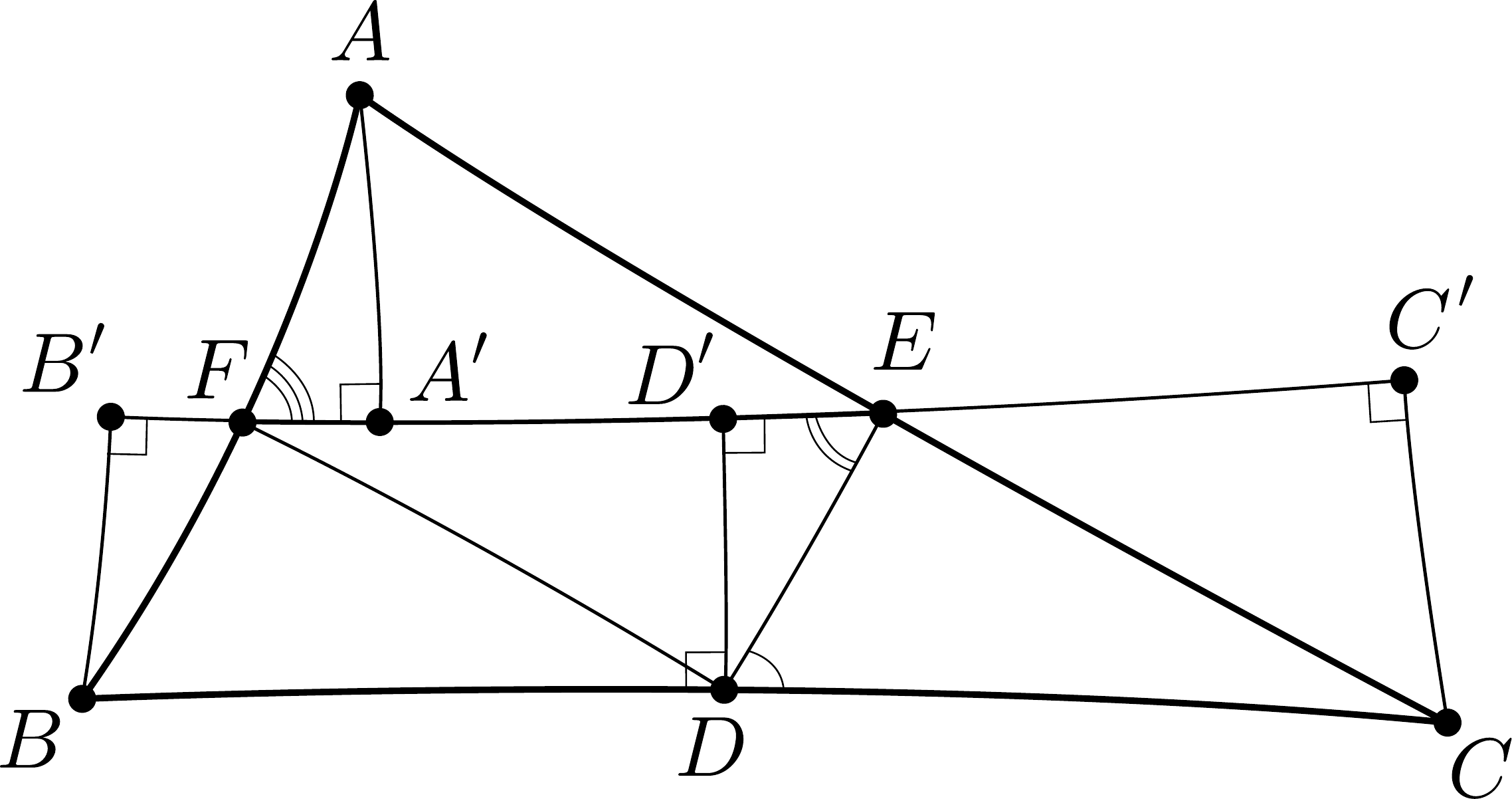}
        	\caption{}
        	\label{fig:last-proof-1}
	\end{center}
\end{figure}

\noindent \textit{Proof of Claim.} In order to allow us to reuse some of our previous constructions and notations, we prove our claim for $\beta_n$ instead of $\alpha_n$. It is enough to show that, for all $n\in\mathbb N$ and any choice of $t_{n+1}$, we have that $\beta_n > \frac{\pi}{2}$ implies $\beta_{n+1} > \frac{\pi}{2}$. Suppose not, and there exists an integer $n$ and a choice of $t_{n+1}$ such that $\beta_n > \frac{\pi}{2}$ but $\beta_{n+1} \leq \frac{\pi}{2}$. There are three possible cases. Let us first investigate the case where $t_{n+1}$ was chosen as the innermost triangle of $t_n$. Following the notation of Figure \ref{fig:construction-both}, where $t_{n+1}=DEF$, we see that if $\beta_{n+1}=\angle FED \leq \frac{\pi}{2}$, $D'$ must be positioned on the closed half-line starting at $E$ and containing the edge $FE$ (see Fig. \ref{fig:last-proof-1}). Since $|B'D'|=|FE|$, this entails that $B'$ lies on the closed half-line starting at $F$ and not containing the open edge $FE$. But since $\beta_n$ is obtuse and the angle at the apex of a hyperbolic Saccheri quadrilateral is acute, $B'$ must in fact lie on the open half-line starting at $F$ and containing the open edge $FE$. A similar contradiction is derived for the other two cases: if $t_{n+1}=EDC$, $D'$ must also be positioned on the half-line starting at $E$ and containing the edge $FE$ and the same contradiction is derived; the case where $t_{n+1}=AFE$ is symmetric with the previous case, exchanging $A$ and $C$.

The exact same argument solves the spherical case, selecting $N\in \mathbb N$ according to Lemma \ref{lemma:saccheri-both} and reversing all the inequalities. To derive the contradiction, this time, we appeal to the fact that $\beta_n$ is acute and that the angle at the apex of a spherical Saccheri quadrilateral is obtuse.

\begin{figure}[H]
	\begin{center}     
    	\includegraphics[width=\textwidth]{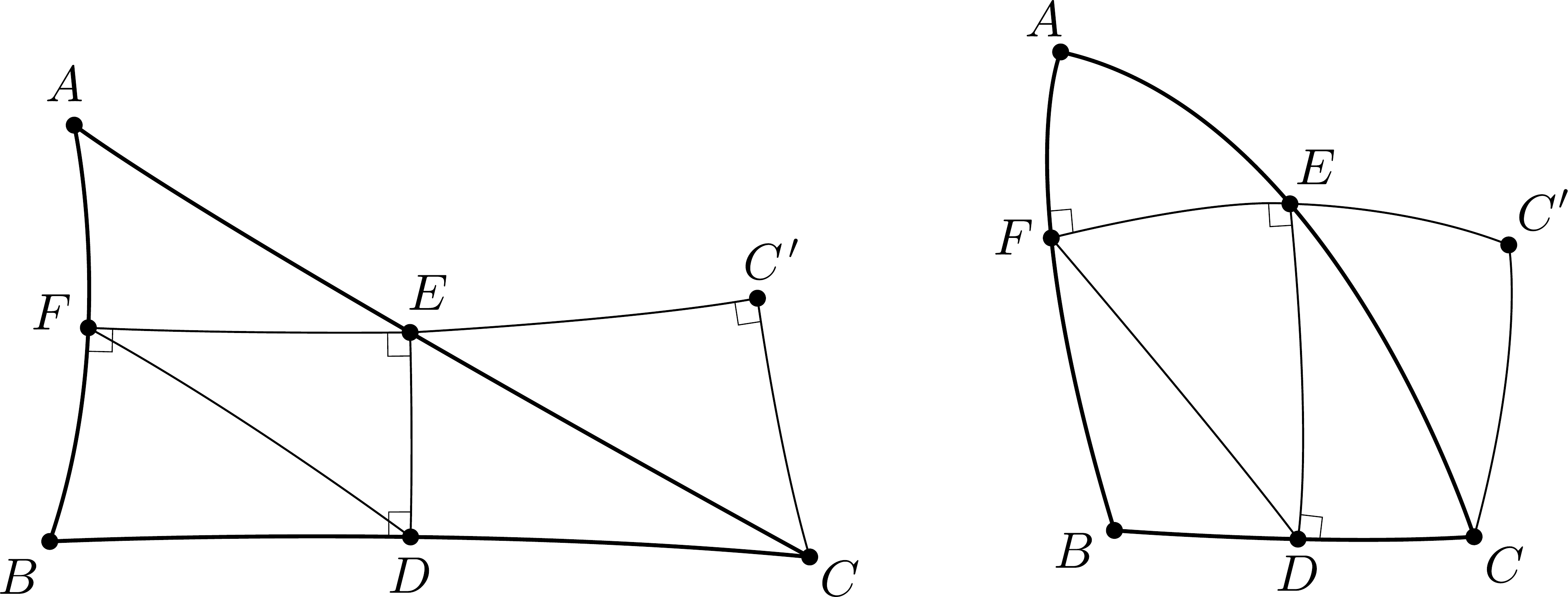}
        	\caption{}
        	\label{fig:last-proof}
	\end{center}
\end{figure}
\begin{rem}
Interestingly, in both the hyperbolic case and the spherical case (as long as $|DD''|\neq\frac{\pi}{2}$), whenever one of the angles $\angle FED$, $\angle EFA$ or $\angle CDE$ is right, they must in fact all be simultaneously right (see Fig. \ref{fig:last-proof}). Indeed, if $\angle EFA$ is right, then $B'=F$ and thus $D'=E$ which forces both other angles to be right. Likewise, if $\angle CDE$ (resp. $\angle FED$) is a right angle, we have $D'=E$ which means that $\angle FED$ (resp. $\angle EFA$) is a right angle and also forces $B'=F$, which implies $\angle EFA$ is also right.
\end{rem}
\end{proof}

As an immediate consequence of Theorem B, we deduce our main theorem:

\vspace{1em}

\noindent\textbf{Theorem A. }\textit{For any geodesic triangle $T$ in $M_\kappa^2$, there exists $\delta>0$ such that, for all $n\in \mathbb N$, all the angles of $T_n$ lie in the interval $(\delta, \pi - \delta)$.}

\begin{proof}
Let $t_0=T_0$ and $\delta_\alpha = \min\{\alpha_0 l_\alpha, (\pi-\alpha_0 L_\alpha)\}$, $\delta_\beta = \min\{\beta_0 l_\beta, (\pi-\beta_0 L_\beta)\}$ and $\delta_\gamma = \min\{\gamma_0 l_\gamma, (\pi-\gamma_0 L_\gamma)\}$. By construction, $\delta\coloneqq \min\{\delta_\alpha, \delta_\beta,\delta_\gamma\} >0$ and every angle of $T_n$ lies in the interval $(\delta, \pi-\delta)$.
\end{proof}
\begin{rem}
\label{alternativeroute}
Both in the hyperbolic and spherical case, there is an easier route to proving Theorem A without the full strength of Theorem B. We briefly explain here why, in both cases, a lower bound on the angles in fact also gives an upper bound. Indeed, in the hyperbolic case, to establish the upper bound in Theorem \ref{thm:main} we need only remember that the sum of the angles of a hyperbolic triangle is always \textit{less} than $\pi$. Therefore if we assume by contradiction that for all $\epsilon > 0$, there exists $N\in \mathbb N$ such that $\alpha_N>\pi - \epsilon$, since for all $n\in \mathbb N$, we have $\beta_n, \gamma_n > \delta$, we are lead to a contradiction for $\epsilon < 2\delta$ as we would then obtain that $\alpha_n + \beta_n +\gamma_n > 2\delta + \pi - 2\delta > \pi$ for some choice of $n$. In the spherical case, the sum of the angles of a triangle is allowed to exceed $\pi$, but the amount by which it does is exactly the area of the triangle. To derive a contradiction we thus need to consider a second parameter to utilise the key property that our triangulations are getting arbitrarily small. Let $A_n$ denote the area of the triangle $t_n$. Suppose then by contradiction that, for all $\epsilon, \epsilon'>0$, there exists $N\in \mathbb N$ such that $\alpha_N>\pi - \epsilon$ and $A_n < \epsilon'$. Let us pick $\epsilon'<2\delta$ and $\epsilon < 2\delta - \epsilon'$. Since for all $n\in \mathbb N$, we have $\beta_n, \gamma_n > \delta$, this leads us to the following contradiction: 

\[
\alpha_n + \beta_n + \gamma_n > 2 \delta + \pi - \epsilon > \pi + \epsilon'
\]

Because of this observation, and using our approach to proving Theorem B, readers interested solely in proving Theorem A need only refer to the proofs of the lower bounds in the proof of Proposition 4.1 and the proofs of the upper bounds in the proof of Theorem C.
\end{rem}


\end{document}